\documentclass[a4paper, 10pt, notitlepage]{article}

\usepackage{amsthm, amsmath, amssymb, latexsym}
\usepackage{mathrsfs,cite}
\usepackage[multiple]{footmisc}
\usepackage{graphicx}
\usepackage[ruled,vlined]{algorithm2e}

\theoremstyle{plain}
\newtheorem{theorem}{Theorem}
\newtheorem{lemma}{Lemma}
\newtheorem{corollary}{Corollary}
\newtheorem{proposition}{Proposition}

\theoremstyle{definition}
\newtheorem{definition}{Definition}
\newtheorem{example}{Example}

\theoremstyle{remark}
\newtheorem{remark}{Remark}

\DeclareMathOperator*{\argmin}{argmin}
\DeclareMathOperator{\dist}{dist}

\author{M.V. Dolgopolik\footnote{Institute for Problems in Mechanical Engineering of the Russian Academy of Sciences, 
Saint Petersburg, Russia}}
\title{Exact penalty functions and global saddle points of augmented Lagrangians for well-posed constrained 
optimization problems}

\begin{document}

\maketitle

\begin{abstract}
The goal of this article is to study necessary and sufficient conditions for the exactness of penalty functions and 
the existence of global saddle points of augmented Lagrangians for well-posed (in a suitable sense) constrained 
optimization problems in infinite dimensional spaces. To this end, we present a new version of extended well-posedness 
of a constrained optimization problem and analyse how it relates to the more well-known types of well-posedness, such as
Tykhonov and Levitin-Polyak well-posedness. This new version of extended well-posedness allows one to extend many
existing results on exact penalty functions and global saddle points of augmented Lagrangians from the finite
dimensional to the infinite dimensional case. Such extensions provide first verifiable sufficient conditions for 
the exactness of penalty functions and the existence of global saddle points of augmented Lagrangians in the infinite
dimensional case that do not rely on very restrictive and difficult to verify assumptions (nonlocal metric regularity
of constraints, existence of nonlocal error bounds, the Palais-Smale condition, abstract properties of the perturbation 
function, etc.) that are typically used in the literature.
\end{abstract}

\section{Introduction}

Exact penalty \cite{MaratosPHD,MaynePolak87,Polak_book,FominyhKarelin2018,LiYu2011,JiangLin2012,DemyanovTamasyan} and
augmented Lagrangian \cite{FortinGlowinski,ItoKunisch,BurmanHansboLarson} methods have been developed for and
successfully applied to various classes of constrained optimization problems in infinite dimensional spaces, including
variational and optimal control problems, problems of computational mechanics, shape optimization, image restoration
problems, etc. (see the aforementioned references). However, a theoretical analysis of \textit{global} properties of
exact penalty functions and augmented Lagrangians for \textit{nonconvex} infinite dimensional problems has remained a
very challenging research topic.

Existing sufficient conditions for the global exactness of penalty functions for general \textit{nonconvex} constrained
optimization problems in infinite dimensional spaces always rely on very restrictive, difficult to verify, and rarely
satisfied assumptions on constraints of the problem, such as the Palais-Smale condition \cite{Zaslavski} and assumptions
on nonlocal metric regularity of constraints/existence of nonlocal error bounds
\cite{Demyanov_Long,Dolgopolik_ExPenI,Dolgopolik_ExPenII}. Sufficient conditions for the existence of global saddle
points of augmented Lagrangians (or, equivalently, for the existence of an augmented Lagrange multiplier/a vector
supporting an exact penalty representation) are typically based either on abstract assumptions on the behaviour of the
problem with respect to perturbations (e.g. the existence of approximately optimal solutions of a perturbed problem
satisfying certain inequalities as in \cite{ShapiroSun}, the validity of certain inequalities for the optimal value of a
perturbed problem as in \cite{ZhouYang2009}, etc.) or on the assumption on the existence of a certain compact
neighbourhood of the feasible region as in \cite{ZhouZhouYang}. The compactness assumption from \cite{ZhouZhouYang} is
rarely satisfied in the infinite dimensional case. In turn, it is unclear how to verify abstract assumptions on
behaviour of the problem with respect to perturbations (such as the ones from \cite{ShapiroSun,ZhouYang2009}) for any
particular class of optimization problems in infinite dimensional spaces.

In the finite dimensional case, \textit{global} properties of penalty functions and augmented Lagrangians (e.g. the
global exactness and the existence of global saddle points) are, roughly speaking, predefined by the \textit{local}
behaviour of these functions near globally optimal solutions of the problem under consideration, according to the
so-called \textit{localization principle}
\cite{Dolgopolik_ExPenII,Dolgopolik_AugmLagrMult,Dolgopolik_AugmLagrTheory,Dolgopolik_UnifedI,Dolgopolik_UnifedII}. This
principle makes it possible to obtain \textit{local} and easily verifiable sufficient conditions for the global
exactness of penalty function and the existence of global saddle points of augmented Lagrangians in the finite
dimensional case. However, as was shown by a number of counterexamples in
\cite{Dolgopolik_ExPenI,Dolgopolik_AugmLagrMult}, the localization principle no longer holds true for infinite
dimensional problems. 

In the general case, it does not seem possible to obtain relatively mild and easy to apply/verify sufficient conditions
for the global exactness of penalty functions and the existence of global saddle points of augmented Lagrangians for
constrained optimization problems in infinite dimensional spaces. Therefore, one needs to select a particular class of
infinite dimensional problems to make the derivation of such conditions more tractable. The class of well-posed (in a
properly chosen sense) problems \cite{DontchevZolezzi} seems to be the most suitable for this purpose. 

The main goal of this article is to show that in the case of well-posed constrained optimization problems in infinite
dimensional spaces one can obtain verifiable sufficient conditions for the global exactness of penalty functions
and the existence of global saddle points of augmented Lagrangians that do not rely on any perturbation-based,
compactness or nonlocal regularity assumptions that are typically used in the literature. We prove that the localization
principle for augmented Lagrangians and exact penalty functions \cite{Dolgopolik_UnifedI,Dolgopolik_UnifedII} admits a
straightforward extension from the finite dimensional case to the case of infinite dimensional well-posed problems.
In particular, we provide first verifiable sufficient conditions for the existence of global saddle points of augmented
Lagrangians for infinite dimensional problems (Theorem~\ref{thrm:AugmLagr_GSP_MainExistenceThrm} and
Corollary~\ref{crlr:GSP_LocPrinciple}) that are not based on any restrictive compactness assumptions or any abstract
assumptions on the behaviour of the problem with respect to perturbations.

In contrast to the sufficient conditions from the aforementioned papers on augmented Lagrangians and exact penalty
functions for infinite dimensional problems, well-posedness is a verifiable property and there exists a vast literature
containing sufficient conditions for the well-posedness of various particular classes of constrained optimization
problems in infinite dimensional spaces \cite{DontchevZolezzi,LucchettiRevalski,Zolezzi81,Zolezzi95,Zolezzi96}.
Moreover, it is well-known that well-posedness is a generic property \cite{Zaslavski2013}, that is, it holds true for
most (in a certain specific topological sense) optimization problems, which, in particular, means that the main results
of this paper, including the localization principle, also hold true generically in the same sense.

The paper is organized as follows. Some auxiliary definitions and results from the theory of exact penalty functions
that are used throughout the paper are collected in Section~\ref{sect:ExactPenalty_Prelim}. Several notions of
well-posedness of constrained optimization problems are studied and compared in Section~\ref{sect:WellPosedness}.
Sufficient conditions for the global exactness of penalty functions for well-posed problems in infinite dimensional
spaces are studied in Section~\ref{sect:ExactPenaltyFunctions}, while sufficient conditions for the existence of global
saddle points of augmented Lagrangians for such problems are analyzed in Section~\ref{sect:AugmentedLagrangians}.

\section{Preliminaries}
\label{sect:ExactPenalty_Prelim}

Let $(X, d)$ be a metric space, $M, Q \subset X$ be nonempty sets, and a function 
$f \colon X \to \mathbb{R} \cup \{ + \infty \}$ be given. Throughout this article we study exact penalty functions 
and augmented Lagrangians for the following optimization problem:
\[
  \min \: f(x) \quad \text{subject to} \quad x \in M, \quad x \in Q.
  \qquad \eqno{(\mathcal{P})}
\]
We suppose that the feasible region of this problem $\Omega := M \cap Q$ is nonempty and its optimal value $f_*$ is
finite. The possibly empty set of globally optimal solutions of the problem $(\mathcal{P})$ is denoted by
$\argmin(\mathcal{P})$.

\begin{remark}
The non-functional constraints $x \in M$ and $x \in Q$ are introduced in order to differentiate between different types
of constraints without explicitly defining these constraints. This separation of constraints onto two different types
provides one with a convenient and flexible theoretical framework that makes all results obtained in this paper
applicable to a wider range of constrained optimization problems. The set $M$ corresponds to those constraints that are
incorporated into the penalty function/augmented Lagrangian, while the set $Q$ is defined by those constraints 
that can/are handled directly. For example, if $X = \mathbb{R}^n$ and the problem under consideration has the form
\[
  \min \: f(x) \quad \text{subject to} \quad F(x) = 0, \quad a \le x \le b,
\]
where $F \colon \mathbb{R}^n \to \mathbb{R}^m$ is a given map, $a, b \in \mathbb{R}^n$, and the inequalities 
$a \le x \le b$ are understood coordinate-wise, one can define $M = \{ x \in \mathbb{R}^n \mid F(x) = 0 \}$ and 
$Q = \{ x \in \mathbb{R}^n \mid a \le x \le b \}$, and consider the penalized problem
\[
  \min \: f(x) + c \| F(x) \| \quad \text{subject to} \quad x \in Q,
\]
where $c > 0$ is the penalty parameter and $\| \cdot \|$ is some norm. However, one can also put 
$M = \{ x \in \mathbb{R}^n \mid F(x) = 0, \: a \le x \le b \}$ and $Q = \mathbb{R}^n$, if one wishes to incorporate all
constraints into a penalty function/augmented Lagrangian.
\end{remark}

Let us recall some definitions and results from the theory of exact penalty functions, most of which can be found 
in \cite{Demyanov_Long,Dolgopolik_ExPenI,Dolgopolik_ExPenII,DolgopolikFominyh}.

\begin{definition}
A function $\varphi \colon X \to [0, + \infty]$ is called \textit{a penalty term} (or \textit{an infeasibility measure})
for the constraint $x \in M$ on the set $Q$, if for any $x \in Q$ one has $\varphi(x) = 0 \iff x \in M$ (or,
equivalently, $\varphi(x) > 0 \iff x \notin M$).
\end{definition}

Let a penalty term $\varphi$ for the constraint $x \in M$ on the set $Q$ be given. The function
\[
  F_c(x) = f(x) + c \varphi(x) \quad \forall x \in X, \: c \ge 0
\]
is called \textit{a penalty function} (for the problem $(\mathcal{P})$). Note that the function $c \mapsto F_c(x)$
is nondecreasing for any $x \in X$ and strictly increasing for any $x \in Q \setminus M$. The problem
\begin{equation} \label{prob:PenalizedProblem}
  \min_x \: F_c(x) \quad \text{subect to} \quad x \in Q
\end{equation}
is called \textit{the penalized problem} (associated with the penalty function $F_c$). 

The main subject of the theory of exact penalty function is the study of interrelations between globally/locally 
optimal solution of the problem $(\mathcal{P})$ and the penalized problem \eqref{prob:PenalizedProblem}.

\begin{definition}
{(i)~The penalty function $F_c$ is called (\textit{globally}) \textit{exact}, if there exists $c_* \ge 0$ such that
for any $c \ge c_*$ globally optimal solutions of the problem $(\mathcal{P})$ and the penalized problem 
\eqref{prob:PenalizedProblem} coincide. The infimum of all such $c_*$ is denoted by $c_*(f, \varphi)$ and is called
\textit{the least exact penalty parameter}.
}

{(ii)~Let $x_*$ be a locally optimal solution of the problem $(\mathcal{P})$. The penalty function $F_c$ is called
\textit{locally exact} at $x_*$, if there exists $c_* \ge 0$ such that for any $c \ge c_*$ the point $x_*$ is 
a locally optimal solution of problem \eqref{prob:PenalizedProblem}. The infimum of all such $c_*$ is denoted by 
$c_*(x_*)$ and is called \textit{the least exact penalty parameter} at $x_*$.
}
\end{definition}

It should be noted that instead of verifying that the sets of globally optimal solutions of the problem 
$(\mathcal{P})$ and the penalized problem \eqref{prob:PenalizedProblem} coincide, it is sufficient to check that only
the optimal values of these problems coincide for some $c > 0$ in order to prove that the penalty function $F_c$
is globally exact (see \cite[Cor.~3.3]{Dolgopolik_ExPenI}).

\begin{proposition} \label{prp:ExactnessViaOptValues}
The penalty function $F_c$ is globally exact if and only if there exists $c \ge 0$ such that
\[
  \inf_{x \in Q} F_c(x) = f_*.
\]
Moreover, the infimum of all $c \ge 0$ satisfying the equality above coincides with $c_*(f, \varphi)$.
\end{proposition}

Denote by $B(x, r) = \{ y \in X \mid d(x, y) \le r \}$ the ball with centre $x \in X$ and radius $r > 0$. Suppose that 
$x_*$ is a locally optimal solution of the problem $(\mathcal{P})$ and the penalty function $F_c$ is locally exact at
$x_*$. Then by definition there exists $c_* \ge 0$ such that for any $c \ge c_*$ the point $x_*$ is a locally optimal
solution of the penalized problem \eqref{prob:PenalizedProblem}. Taking into account the fact that the function
$c \mapsto F_c(x)$ is nondecreasing for any $x \in X$ and $F_c(x_*) = f(x_*)$, one can conclude that there exists 
$r > 0$ such that
\[
  F_c(x) \ge F_c(x_*) \quad \forall x \in B(x_*, r) \cap Q \quad \forall c \ge c_*
\]
(note that $r > 0$ is the same for all $c \ge c_*$). Denote by $r_*(x_*, c_*)$ the supremum of all $r > 0$ for which
the inequality above holds true. The quantity $r_*(x_*, c_*) > 0$ obviously depends on the chosen value of the penalty
parameter $c_*$. It is natural to refer to $r_*(x_*, c_*)$ as \textit{the radius of local exactness} of the penalty 
function $F_c$ at the point $x_*$ associated with the penalty parameter $c_* \ge 0$. Below we will need the radius of 
local exactness to define uniform local exactness of a penalty function on a set of locally optimal solutions.

Since to the best of the author's knowledge the radius of local exactness has not been properly studied in the 
literature, let us point out a way it can be estimated in many cases with the use of some standard results and
techniques from the theory of exact penalty functions and variational analysis (see, e.g. 
\cite[Thm.~2.4, Prop.~2.7, and Cor.~2.8]{Dolgopolik_ExPenI}).

For any $V \subset X$ and $x \in X$ denote by $\dist(x, V) = \inf_{y \in V} d(x, y)$ the distance between $x$ and $V$.

\begin{proposition}
Let $x_* \in X$ be a locally optimal solution of the problem $(\mathcal{P})$ and $r_1 > 0$ be such that
$f(x) \ge f(x_*)$ for all $x \in B(x_*, r_1) \cap \Omega$. Suppose also that for some $r_2 > 0$ the function $f$ is 
H\"{o}lder continuous on the ball $B(x_*, r_2)$ with exponent $\alpha > 0$ and constant $L > 0$, that is,
\begin{equation} \label{eq:HolderContinuity}
  |f(x) - f(y)| \le L d(x, y)^{\alpha} \quad \forall x, y \in B(x_*, r_2).
\end{equation}
Suppose finally that there exist $a > 0$ and $r_3 > 0$ such that
\begin{equation} \label{eq:PenTerm_LocErrorBound}
  \varphi(x) \ge a \dist(x, \Omega)^{\alpha} \quad \forall x \in B(y, r_3).
\end{equation}
Then the penalty function $F_c$ is locally exact at $x_*$, $c_*(x_*) \le L / a$, and the inequality
$r(x_*, L/a) \ge r_* := \min\{ r_1/2, r_2, r_3 \}$ holds true.
\end{proposition}

\begin{proof}
Fix any $x \in B(x_*, r_*) \cap Q$. If $x \in M$, then $x$ is feasible for the problem $(\mathcal{P})$ and by 
the definitions of $r_*$ and $r_1$ one has $F_c(x) = f(x) \ge f(x_*) = F_c(x_*)$ for any $c \ge 0$.

Suppose now that $x \notin M$. By the definition of the distance between a point and a set one can find
a sequence $\{ x_n \} \subset \Omega$ such that $d(x, x_n) \to \dist(x, \Omega)$ as $n \to \infty$ and
\[
  d(x, x_n) \le d(x, x_*) \le r_* \le \frac{r_1}{2}, \quad 
  d(x, x_n) \ge d(x, x_{n + 1}) \quad \forall n \in \mathbb{N}.
\]
Note that
\[
  d(x_n, x_*) \le d(x_n, x) + d(x, x_*) \le \frac{r_1}{2} + \frac{r_1}{2} = r_1,
\]
which means that $f(x_n) \ge f(x_*)$ for all $n \in \mathbb{N}$. Hence with the use of \eqref{eq:HolderContinuity}
one obtains that
\[
  f(x_*) - f(x) \le f(x_*) - f(x_n) + f(x_n) - f(x) \le f(x_n) - f(x) \le L d(x_n, x)^{\alpha}
\]
for all $n \in \mathbb{N}$ (here we used the fact that $d(x_n, x) \le r_* \le r_2$). Passing to the limit 
as $n \to \infty$ one can conclude that $f(x) \ge f(x_*) - L \dist(x, \Omega)^{\alpha}$. Now, applying inequality
\eqref{eq:PenTerm_LocErrorBound} one finally gets that
\[
  F_c(x) = f(x) + c \varphi(x) \ge f(x_*) - L \dist(x, \Omega)^{\alpha} + c a \dist(x, \Omega)^{\alpha}
  \ge f(x_*) = F_c(x_*)
\]
for any $c \ge L / a$. Since $x \in B(x_*, r_*) \cap Q$ was chosen arbitrarily, one can conclude that
\[
  F_c(x) \ge F_c(x_*) \quad \forall x \in B(x_*, r_*) \quad \forall c \ge \frac{L}{a},
\]
which means that the penalty function $F_c$ is locally exact at $x_*$, $c_*(x_*) \le L / a$ and
$r(x_*, L/a) \ge r_*$.
\end{proof}

\begin{remark}
Note that assumption \eqref{eq:PenTerm_LocErrorBound} simply means that the penalty term $\varphi$ admits a H\"{o}lder
error bound with modulus $a > 0$ near $x_*$ (see \cite{LiMordukhovich,KrugerLopezYangZhu}).
\end{remark}

\section{Extended well-posedness of optimization problems}
\label{sect:WellPosedness}

Let us introduce several definitions of generalized well-posedness of optimization problems. They extend the classic
concepts of the Tykhonov \cite{Tykhonov63,Tykhonov66} and the Levitin-Polyak \cite{LevitinPolyak,Polyak} well-posedness
of optimization problems to the case of problems with non-unique global minimizers, and are closely related to the
extended well-posedness \cite{DontchevZolezzi,Zolezzi96} and the generalized Levitin-Polyak well-posedness
\cite{HuangYang}.

Recall that $\Omega$ is the feasible region of the problem $(\mathcal{P})$.

\begin{definition} \label{def:TykhonovWellPosed}
{(i)~The problem $(\mathcal{P})$ is called \textit{Tykhonov well-posed in the extended sense}, if the set
$\argmin(\mathcal{P})$ is nonempty and for any sequence $\{ x_n \} \subset \Omega$ such that $f(x_n) \to f_*$ as 
$n \to \infty$ there exists a subsequence $\{ x_{n_k} \}$ converging to some $x_* \in \argmin(\mathcal{P})$.
}

\noindent{(ii)~The problem $(\mathcal{P})$ is called \textit{weakly Tykhonov well-posed in the extended sense}, if 
the set $\argmin(\mathcal{P})$ is nonempty and for any sequence $\{ x_n \} \subset \Omega$ such that $f(x_n) \to f_*$ 
as $n \to \infty$ one has $\dist(x_n, \argmin(\mathcal{P})) \to 0$ as $n \to \infty$.
}
\end{definition}

In the case when the set $\argmin(\mathcal{P})$ is a singleton (i.e. a global minimizer of $(\mathcal{P})$ is unique),
the two definitions above coincide with the standard definition of Tykhonov well-posedness: there exists a unique global
minimizer $x_*$ of the problem $(\mathcal{P})$ and for any sequence $\{ x_n \} \subset \Omega$ such that 
$f(x_n) \to f_*$ one has $x_n \to x_*$ as $n \to \infty$. Let us show how the two definition are related to each other
and the compactness of the set of globally optimal solutions of the problem $(\mathcal{P})$.

\begin{proposition} \label{prp:TykhonovWellPosedness_2Versions}
The problem $(\mathcal{P})$ is Tykhonov well-posed in the extended sense if and only if it is weakly Tykhonov well-posed
in the extended sense and the set $\argmin(\mathcal{P})$ is compact.
\end{proposition}

\begin{proof}
\textbf{Part 1.} Let the problem $(\mathcal{P})$ be Tykhonov well-posed in the extended sense. Choose any sequence 
$\{ x_n \} \subset \argmin(\mathcal{P})$. Then $f(x_n) \equiv f_* \to f_*$ as $n \to \infty$, which by the first part 
of Definition~\ref{def:TykhonovWellPosed} means that there exists a subsequence $x_{n_k}$ converging to some 
$x_* \in \argmin(\mathcal{P})$. Thus, the set $\argmin(\mathcal{P})$ is compact.

Suppose by contradiction that the problem $(\mathcal{P})$ is not weakly Tykhonov well-posed in the extended sense. Then
there exists a sequence $\{ x_n \} \subset \Omega$ such that $f(x_n) \to f_*$ as $n \to \infty$, but the sequence 
$\{ \dist(x_n, \argmin(\mathcal{P})) \}$ does not converge to zero. Therefore one can find $\varepsilon > 0$ and
a subsequence $\{ x_{n_k} \}$ such that 
\begin{equation} \label{eq:TykhonovNonWeaklyWellPosedSeq}
  \dist\big( x_{n_k}, \argmin(\mathcal{P}) \big) \ge \varepsilon \quad \forall k \in \mathbb{N}. 
\end{equation}
Note that $f(x_{n_k}) \to f_*$ as $k \to \infty$, which by the definition of Tykhonov well-posedness in the extended
sense means that there exists a subsequence of the sequence $\{ x_{n_k} \}$, which we denote again by $\{ x_{n_k} \}$,
converging to some $x_* \in \argmin(\mathcal{P})$. Since $d(x_{n_k}, x_*) \ge \dist(x_{n_k}, \argmin(\mathcal{P}))$,
one can conclude that $\dist(x_{n_k}, \argmin(\mathcal{P})) \to 0$ as $k \to \infty$, which contradicts
\eqref{eq:TykhonovNonWeaklyWellPosedSeq}.

\textbf{Part 2.} Let us prove the converse statement. Suppose that the problem $(\mathcal{P})$ is weakly Tykhonov 
well-posed in the extended sense and the set $\argmin(\mathcal{P})$ is compact. Choose any sequence 
$\{ x_n \} \subset \Omega$ such that $f(x_n) \to f_*$ as $n \to \infty$. Since the problem $(\mathcal{P})$ is weakly 
Tykhonov well-posed in the extended sense, one has $\dist(x_n, \argmin(\mathcal{P})) \to 0$ as $n \to \infty$. Denote 
$t_n = \dist(x_n, \argmin(\mathcal{P}))$. 

If an infinite number of elements of the sequence $\{ t_n \}$ is equal to zero, then by virtue of the compactness 
(and, therefore, closedness) of the set $\argmin(\mathcal{P})$ there exists a subsequence $\{ x_{n_k} \}$ such that 
$\{ x_{n_k} \} \subset \argmin(\mathcal{P})$. Furthermore, applying the compactness of the set $\argmin(\mathcal{P})$
again one can extract a subsequence of this sequence, which we denote again by $\{ x_{n_k} \}$, converging to some 
$x_* \in \argmin(\mathcal{P})$.

Suppose now that only a finite number of elements of the sequence $\{ t_n \}$ is equal to zero. Without loss of
generality one can suppose that all $t_n$ are not equal to zero. By the definition of the distance between a point 
and a set, for any $n \in \mathbb{N}$ one can find $y_n \in \argmin(\mathcal{P})$ such that $d(x_n, y_n) \le 2 t_n$. 
Due to the compactness of $\argmin(\mathcal{P})$ one can extract a subsequence $\{ y_{n_k} \}$ converging to some 
$x_* \in \argmin(\mathcal{P})$. Since $d(x_n, y_n) \le 2 t_n$ and $t_n \to 0$ as $n \to \infty$, the corresponding 
subsequence $\{ x_{n_k} \}$ also converges to $x_*$.

Thus, in both cases we have found a subsequence $\{ x_{n_k} \}$ converging to some $x_* \in \argmin(\mathcal{P})$,
which by definition means that the problem $(\mathcal{P})$ is Tykhonov well-posed in the extended sense.
\end{proof} 

Next we introduce extended versions of the Levitin-Polyak well-posedness, which expand the notion of Tykhonov
well-posedness to the case of potentially infeasible minimizing sequences $\{ x_n \}$.

\begin{definition} \label{def:LevitinPolyakWellPosed}
{(i)~The problem $(\mathcal{P})$ is called \textit{Levitin-Polyak well-posed in the extended sense}, if the set
$\argmin(\mathcal{P})$ is nonempty and for any sequence $\{ x_n \} \subset X$ such that $f(x_n) \to f_*$ and 
$\dist(x_n, \Omega) \to 0$ as $n \to \infty$ there exists a subsequence $\{ x_{n_k} \}$ converging to some 
$x_* \in \argmin(\mathcal{P})$.
}

\noindent{(ii)~The problem $(\mathcal{P})$ is called \textit{weakly Levitin-Polyak well-posed in the extended sense}, 
if the set $\argmin(\mathcal{P})$ is nonempty and for any sequence $\{ x_n \} \subset X$ such that $f(x_n) \to f_*$ and
$\dist(x_n, \Omega) \to 0$ as $n \to \infty$ one has $\dist(x_n, \argmin(\mathcal{P})) \to 0$ as $n \to \infty$.
}
\end{definition}

In the case when a globally optimal solution of the problem $(\mathcal{P})$ is unique, both parts of the definition
above are reduced to the standard definition of the Levitin-Polyak well-posedness \cite{DontchevZolezzi}. In addition,
arguing in the same way as in the proof of Proposition~\ref{prp:TykhonovWellPosedness_2Versions} one can easily verify
that the following result holds true.

\begin{proposition} \label{prp:LevitinPolyakWellPosedness_2Versions}
The problem $(\mathcal{P})$ is Levitin-Polyak well-posed in the extended sense if and only if it is weakly
Levitin-Polyak well-posed in the extended sense and the set $\argmin(\mathcal{P})$ is compact.
\end{proposition}

Let us provide simple and easily verifiable in many important cases sufficient conditions for the Levitin-Polyak
well-posedness in the extended sense to be equivalent to the Tykhonov well-posedness in the extended sense. 

\begin{theorem}
Let $X$ be a real normed space and $f$ be uniformly continuous (in particular, Lipschitz continuous) on bounded sets.
Then for the problem $(\mathcal{P})$ to be weakly Levitin-Polyak well-posed in the extended sense it is sufficient that
this problem is weakly Tykhonov well-posed in the extended sense and
\begin{equation} \label{eq:LPWellPosedCond}
  \liminf_{(R, \delta) \to (+ \infty, 0)} \inf\Big\{ f(x) \Bigm| 
  x \in X \setminus \Omega \colon \dist(x, \Omega) < \delta, \: \| x \| \ge R \Big\} > f_*.
\end{equation}
Moreover, these conditions become necessary for the problem $(\mathcal{P})$ to be weakly Levitin-Polyak well-posed in
the extended sense, if the set $\argmin(\mathcal{P})$ is bounded.
\end{theorem}

\begin{proof}
\textbf{Sufficiency.} Let the problem $(\mathcal{P})$ be weakly Tykhonov well-posed in the extended sense and condition
\eqref{eq:LPWellPosedCond} hold true. Suppose by contradiction that the problem $(\mathcal{P})$ is not weakly
Levitin-Polyak well-posed in the extneded sense. Then by definition one can find a sequence $\{ x_n \} \subset X$ such
that $f(x_n) \to f_*$ and $\dist(x_n, \Omega) \to 0$ as $n \to \infty$, but the distance
$\dist(x_n, \argmin(\mathcal{P}))$ does not converge to zero. Replacing, if necessary, the sequence $\{ x_n \}$ by 
a subsequence, one can suppose that there exists $\varepsilon > 0$ such that 
\begin{equation} \label{eq:LPNonWellPosedness}
  \dist(x_n, \argmin(\mathcal{P})) \ge \varepsilon \quad \forall n \in \mathbb{N}.
\end{equation}
Furthermore, due to the weak Tykhonov well-posedness in the extended sense one can also suppose that 
$\{ x_n \} \subset X \setminus \Omega$.

From the facts that $f(x_n) \to f_*$ and $\dist(x_n, \Omega) \to 0$ as $n \to \infty$ and condition
\eqref{eq:LPWellPosedCond} holds true it follows that the sequence $\{ x_n \}$ is bounded. Denote 
$R = \sup_n \| x_n \|$. By our assumption the function $f$ is uniformly continuous on $B(0, R + 1)$, which means that
for any $k \in \mathbb{N}$ there exists $\delta_k > 0$ such that $|f(x) - f(y)| < 1 / (k + 1)$ for any 
$x, y \in B(0, R + 1)$ satisfying the inequality $\| x - y \| < \delta_k$. Clearly, one can assume that $\delta_k \to 0$
as $k \to \infty$

Fix any $k \in \mathbb{N}$. Since $\dist(x_n, \Omega) \to 0$ as $n \to \infty$, one can find $n_k \in \mathbb{N}$ such
that $\dist(x_{n_k}, \Omega) < \min\{ \delta_k, 1 \}$. Consequently, there exists $y_k \in \Omega \cap B(0, R + 1)$ such
that 
\begin{equation} \label{eq:LPNonWellPosedSeqSister}
  \| x_{n_k} - y_k \| < \delta_k,
\end{equation}
which by the choice of $\delta_k$ implies that $|f(x_{n_k}) - f(y_k)| < 1 / (k + 1)$. Recall that $f(x_n) \to f_*$ as 
$n \to \infty$. Therefore $f(y_k) \to f_*$ as well. Moreover, $\{ y_k \} \subset \Omega$. Therefore 
$\dist(y_k, \argmin(\mathcal{P})) \to 0$ as $k \to \infty$ due to the fact that the problem $(\mathcal{P})$ is weakly
Tykhonov well-posed in the extended sense. Hence with the use of \eqref{eq:LPNonWellPosedSeqSister}
one can conclude that $\dist(x_{n_k}, \argmin(\mathcal{P})) \to 0$ as $k \to \infty$, which contradicts
\eqref{eq:LPNonWellPosedness}.

\textbf{Necessity.} Suppose now that the set $\argmin(\mathcal{P})$ is bounded and the problem $(\mathcal{P})$ is
weakly Levitin-Polyak well-posed in the extended sense. Then this problem is obviously weakly Tykhonov well-posed in 
the extended sense as well (see Definitions~\ref{def:TykhonovWellPosed} and \ref{def:LevitinPolyakWellPosed}). Suppose
by contradiction that condition \eqref{eq:LPWellPosedCond} is not satisfied. Then 
\[
  \liminf_{(R, \delta) \to (+ \infty, 0)} \inf\Big\{ f(x) \Bigm| 
  x \in X \setminus \Omega \colon \dist(x, \Omega) < \delta, \: \| x \| \ge R \Big\} \le f_*,
\]
which means that there exists a sequence $\{ x_n \} \subset X \setminus \Omega$ such that $\dist(x_n, \Omega) \to 0$ and
$\| x_n \| \to + \infty$ as $n \to \infty$, and $\liminf_{n \to \infty} f(x_n) \le f_*$. Replacing, if necessary, the
sequence $\{ x_n \}$ by a subsequence, one can suppose that
\[
  \underline{f} := \liminf_{n \to \infty} f(x_n) = \lim_{n \to \infty} f(x_n).
\]
If $\underline{f} = f_*$, then due to the weak Levitin-Polyak well-posedness in the extended sense one has 
$\dist(x_n, \argmin(\mathcal{P})) \to 0$ as $n \to \infty$, which is impossible by virtue of the facts that the set
$\argmin(\mathcal{P})$ is bounded, while $\| x_n \|$ increases unboundedly. Therefore, we can assume that 
$\underline{f} < f_*$.

Recall that $\dist(x_n, \Omega) \to 0$ as $n \to \infty$. Consequently, for any $n \in \mathbb{N}$ there exists 
$y_n \in \Omega$ such that $\| x_n - y_n \| \to 0$ as $n \to \infty$. For any $n \in \mathbb{N}$ introduce the function
$g_n(t) = f(t x_n + (1 - t) y_n)$, $t \in [0, 1]$. Note that the function $f$ is continuous on $X$, since it is
uniformly continuous on bounded sets. Therefore, all functions $g_n$ are continuous. Moreover, 
$g_n(0) = f(y_n) \ge f_*$, due to the fact that $y_n \in \Omega$, while $g_n(1) = f(x_n) \to \underline{f}$ as 
$n \to \infty$. Hence $g_n(1) < f_*$ for any sufficiently large $n \in \mathbb{N}$, and by the intermediate value
theorem for any such $n$ one can find $t_n \in [0, 1)$ such that $g_n(t_n) = f_*$. 

Denote $z_n = t_n x_n + (1 - t_n) y_n$. By construction $f(z_n) = g_n(t_n) = f_*$ for any sufficiently large $n$.
Furthermore, since $\{ y_n \} \subset \Omega$ and $\dist(x_n, \Omega) \to 0$ as $n \to \infty$, one also has 
$\dist(z_n, \Omega) \to 0$ as $n \to \infty$. Therefore, by the weak Levitin-Polyak well-posedness in the extended sense
one has $\dist(z_n, \argmin(\mathcal{P})) \to 0$ as $n \to \infty$. However, this is impossible due to the facts that
the set $\argmin(\mathcal{P})$ is bounded, but $\| z_n \| \to + \infty$ as $n \to \infty$, since 
$z_n = t_n x_n + (1 - t_n) y_n$ for some $t_n \in [0, 1]$, and $\| x_n \| \to + \infty$ and $\| x_n - y_n \| \to 0$ as
$n \to \infty$.
\end{proof}

\begin{corollary}
Let $X$ be a real normed space, $f$ be uniformly continuous on bounded sets, and the feasible set $\Omega$ be bounded.
Then the problem $(\mathcal{P})$ is (weakly) Levitin-Polyak well-posed in the extended sense if and only if it is
(weakly) Tykhonov well-posed in the extended sense.
\end{corollary}

\begin{corollary}
Let $X$ be a real normed space and $f$ be uniformly continuous on bounded sets. Suppose that there exists a
nondecreasing function $\omega \colon [0, + \infty) \to [0, + \infty)$ satisfying the following two conditions:
\begin{enumerate}
\item{$\omega(t) = 0$ if and only if $t = 0$.}

\item{$f(x) + \omega(\dist(x, \Omega)) \to + \infty$ as $\| x \| \to + \infty$.}
\end{enumerate}
Then the problem $(\mathcal{P})$ is (weakly) Levitin-Polyak well-posed in the extended sense if and only if it is
(weakly) Tykhonov well-posed in the extended sense.
\end{corollary}

In the context of duality theory and primal-dual methods, one often deals with sequences $\{ x_n \}$ for which it is
difficult to prove the convergence of the distance $\dist(x_n, \Omega)$ to zero, but the convergence of some natural
infeasibility measure to zero can be easily verified. Being inspired by the ideas of Huang and Yang \cite{HuangYang}, 
below we introduce the definition of generalized well-posedness that it is tailored to this specific context. 

Suppose that some infeasibility measure $\varphi$ for the constraint $x \in M$ on the set $Q$ is given. Recall that it
means that $\varphi \colon X \to [0, + \infty]$ and for any $x \in Q$ one has $\varphi(x) = 0 \iff x \in M$, that is,
for any point that violates the constraint $x \in M$ (i.e. $x \notin M$) the value $\varphi(x)$ is positive and can be
viewed as an infeasibility measure for the constraint $x \in M$. In particular, if $M$ is defined by the constraint
$G(x) \in K$, that is, $M = \{ x \in X \mid G(x) \in K \}$ for some mapping $G \colon X \to Y$ and some nonempty set $K
\subset Y$, where $Y$ is a metric space, one can define $\varphi(x) = \dist(G(x), K)$.

\begin{definition} \label{def:LevitinPolyakWellPosed_InfMeas}
{(i)~The problem $(\mathcal{P})$ is called \textit{Levitin-Polyak well-posed with respect to the infeasibility measure
$\varphi$}, if the set $\argmin(\mathcal{P})$ is nonempty and for any sequence $\{ x_n \} \subset X$ such that 
$f(x_n) \to f_*$ and $\varphi(x_n) \to 0$ as $n \to \infty$ there exists a subsequence $\{ x_{n_k} \}$ converging to
some $x_* \in \argmin(\mathcal{P})$.
}

\noindent{(ii)~The problem $(\mathcal{P})$ is called \textit{weakly Levitin-Polyak well-posed with respect to the
infeasibility measure $\varphi$}, if the set $\argmin(\mathcal{P})$ is nonempty and for any sequence 
$\{ x_n \} \subset X$ such that $f(x_n) \to f_*$ and $\varphi(x_n) \to 0$ as $n \to \infty$ one has 
$\dist(x_n, \argmin(\mathcal{P})) \to 0$ as $n \to \infty$.
}
\end{definition}

\begin{remark}
Note that in the case when the feasible region $\Omega = M \cap Q$ of the problem $(\mathcal{P})$ is closed, the
Levitin-Polyak well-posedness in the extended sense coincides with the Levitin-Polyak well-posedness with respect to
the infeasibility measure $\varphi(\cdot) := \dist(\cdot, \Omega)$.
\end{remark}

Let us first point out how the two parts of this definition are connected to each other. The proof of the following
result is similar to the proof of Proposition~\ref{prp:TykhonovWellPosedness_2Versions}

\begin{proposition} \label{prp:GenLevitinPolyakWellPosedness_2Versions}
The problem $(\mathcal{P})$ is Levitin-Polyak well-posed with respect to the infeasibility measure $\varphi$ if and 
only if it is weakly Levitin-Polyak well-posed with respect to the infeasibility measure $\varphi$ and the set 
$\argmin(\mathcal{P})$ is compact.
\end{proposition}

Let us also show that the Levitin-Polyak well-posedness with respect to an infeasibility measure is implied by 
the Levitin-Polyak well-posedness in the extended sense, provided certain (nonlinear) nonlocal error bounds hold. It
should be noted that such nonlocal error bounds are often explicitly or implicitly used in various sufficient conditions
for the global exactness of penalty functions in the infinite dimensional case (see, e.g. 
\cite{Demyanov_Long,Dolgopolik_ExPenI,DolgopolikFominyh,DemyanovGiannessiKarelin,Demyanov2003,DemyanovGiannessi,
DemyanovGiannessiTamasyan,Karelin,Zaslavski}). Furthermore, under some additional assumptions one can prove 
the validity of such nonlocal error bounds in some particular cases, e.g. convex case \cite{Deng98,WangPang}, 
DC case \cite{LeThi}, piecewise affine case \cite{Dolgopolik_PieceAff}, polynomial case
\cite{Li2013,Vui2013,DinhHaPham}, etc.

\begin{proposition}
Let there exist $\delta > 0$ and a continuous nondecreasing function $\omega \colon [0, + \infty) \to [0, + \infty)$ 
such that $\omega(t) = 0$ if and only if $t = 0$ and
\begin{equation} \label{eq:NonlocalErrorBound}
  \varphi(x) \ge \omega(\dist(x, \Omega)) \quad \forall x \in X \setminus \Omega \colon \varphi(x) < \delta.
\end{equation}
Then for the problem $(\mathcal{P})$ to be weakly Levitin-Polyak well-posed with respect to the infeasibility measure
$\varphi$ it is sufficient that this problem is weakly Levitin-Polyak well-posed in the extended sense.
\end{proposition}

\begin{proof}
The validity of the claim of this proposition follows directly from Definition~\ref{def:LevitinPolyakWellPosed} and 
\ref{def:LevitinPolyakWellPosed_InfMeas} and the fact that by inequality \eqref{eq:NonlocalErrorBound}, if for some
sequence $\{ x_n \} \subset X$ one has $\varphi(x_n) \to 0$ as $n \to \infty$, then $\dist(x_n, \Omega) \to 0$ as 
$n \to \infty$ as well.
\end{proof}

\begin{proposition}
Let $X$ be a real normed space and the following two conditions be valid:
\begin{enumerate}
\item{there exists a nondecreasing function $\eta \colon[0, + \infty) \to [0, + \infty)$ such that $\eta(t) = 0$ if 
and only if $t = 0$ and $f(x) + \eta(\varphi(x)) \to + \infty$ as $\| x \| \to + \infty$;
}

\item{for any $R > 0$ there exist $\delta_R > 0$ and a continuous nondecreasing function 
$\omega_R \colon [0, + \infty) \to [0, + \infty)$ such that $\omega_R(t) = 0$ if and only if $t = 0$ and
\[
  \varphi(x) \ge \omega_R(\dist(x, \Omega)) 
  \quad \forall x \in B(0, R) \setminus \Omega \colon \varphi(x) < \delta_R.
\]
}
\end{enumerate}
Then for the problem $(\mathcal{P})$ to be weakly Levitin-Polyak well-posed with respect to the infeasibility measure
$\varphi$ it is sufficient that this problem is weakly Levitin-Polyak well-posed in the extended sense.
\end{proposition}

\begin{proof}
The first condition of the proposition implies that if for some sequence $\{ x_n \} \subset X$ one has $f(x_n) \to f_*$
and $\varphi(x_n) \to 0$ as $n \to \infty$, then the sequence $\{ x_n \}$ is bounded. Hence with use of the second
condition one can conclude that for any such sequence $\{ x_n \}$ one has $\dist(x_n, \Omega) \to 0$ as $n \to \infty$,
which obviously implies the required result.
\end{proof}

Under relatively mild assumptions on the function $\varphi$ one can show that the opposite results to the two previous 
propositions hold true as well.

\begin{proposition}
Let $X$ be a real normed space and one of the two following conditions hold true:
\begin{enumerate}
\item{$\varphi$ is uniformly continuous on the set $\{ x \in X \mid \dist(x, \Omega) < \delta \}$ 
for some $\delta > 0$;}

\item{$\varphi$ is uniformly continuous on bounded sets and one of the following three conditions are satisfied:
  \begin{enumerate}
    \item{the set $\Omega$ is bounded,}
    
    \item{there exists a nondecreasing function $\eta \colon [0, + \infty) \to [0, + \infty)$ such that
    $f(x) + \eta(\dist(x, \Omega)) \to + \infty$ as $\| x \| \to + \infty$,}
    
    \item{the set $\{ x \in X \mid |f(x) - f_*| < \alpha, \: \dist(x, \Omega) < \delta \}$ is bounded for some 
    $\alpha > 0$ and $\delta > 0$.}
  \end{enumerate}
}
\end{enumerate}
Then for the problem $(\mathcal{P})$ to be weakly Levitin-Polyak well-posed in the extended sense it is sufficient 
that this problem is weakly Levitin-Polyak well-posed with respect to the infeasibility measure $\varphi$.
\end{proposition}

\begin{proof}
Let the problem $(\mathcal{P})$ be weakly Levitin-Polyak well-posed with respect to the infeasibility measure 
$\varphi$. Suppose that a sequence $\{ x_n \} \subset X$ is such that $f(x_n) \to f_*$ and $\dist(x_n, \Omega) \to 0$ 
as $n \to \infty$. If the first assumption holds true, then taking into account the fact that $\varphi(x) = 0$ for 
any $x \in \Omega$ one can conclude that for any $\varepsilon > 0$ there exists $r \in (0, \delta)$ such that 
$\varphi(x) < \varepsilon$ for any $x \in X$ satisfying the inequality $\dist(x, \Omega) < r$. Therefore, for any 
sufficiently large $n$ one has $\varphi(x_n) < \varepsilon$, which implies that $\varphi(x_n) \to 0$ as 
$n \to \infty$. Hence by Definition~\ref{def:LevitinPolyakWellPosed_InfMeas} one can conclude that 
$\dist(x_n, \argmin(\mathcal{P})) \to 0$ as $n \to \infty$, which means that the problem $(\mathcal{P})$ is weakly
Levitin-Polyak well-posed in the extended sense.

The proof of the claim of the proposition in the case when $\varphi$ is uniformly continuous on bounded sets
almost literally repeats the proof of the first case. One only needs to note that each of the three additional 
assumptions ensures that any sequence $\{ x_n \} \subset X$ such that $f(x_n) \to f_*$ and $\dist(x_n, \Omega) \to 0$ 
as $n \to \infty$ is necessarily bounded.
\end{proof}

\begin{remark}
For the sake of completeness let us note that both the Levitin-Polyak well-posedness in the extended sense and
the Levitin-Polyak well-posedness with respect to any infeasibility measure imply Tykhonov well-posedness. The converse
statement obviously does not hold true in the general case.
\end{remark}

\section{Exact penalty functions for well-posed problems}
\label{sect:ExactPenaltyFunctions}

As was noted in the introduction, the localization principle for exact penalty functions/augmented Lagrangians no 
longer holds true in the general infinite dimensional case and one typically has to employ restrictive assumptions, 
such as semiglobal error bounds, semiglobal metric regularity of constraints, the Palais-Smail condition, etc., 
in order to prove global exactness of penalty functions in this case. Our aim is to prove that in the infinite 
dimensional case the localization principle is valid for well-posed problems. To this end, we need to first recall 
the definition of the zero duality gap property \cite{Jeyakumar,RubinovHuangYang}.

Let, as above, $\varphi \colon X \to [0, + \infty]$ be an infeasibility measure/penalty term for the constraint
$x \in M$ on the set $Q$. Define 
\[
  F_c(x) = f(x) + c \varphi(x), \quad \Theta(c) := \inf_{x \in Q} F_c(x) \quad \forall c \ge 0.
\]
The function $\Theta$ is called \textit{the dual function} for the penalty function $F_c$. The function $\Theta$ 
is obviously non-decreasing and concave. Furthermore, it has a non-empty effective domain if and only if $F_c$ is 
bounded below on $Q$ for some $c \ge 0$. The problem 
\begin{equation} \label{eq:DualProblem_Penalty}
  \max_{c \ge 0} \: \Theta(c)
\end{equation}
is called \textit{the dual problem} (associated with the penalty function $F_c$). Its optimal value is denoted by
$\Theta_*$.

From the fact that $F_c(x) = f(x)$ for any feasible point $x$ it follows that
\[
  \Theta(c) = \inf_{x \in Q} F_c(x) \le \inf_{x \in \Omega} F_c(x) = \inf_{x \in \Omega} f(x) = f_*,
\]
which implies that $\Theta_* \le f_*$, that is, \textit{the weak duality} (between the problem $(\mathcal{P})$ and dual
problem \eqref{eq:DualProblem_Penalty}) holds true. The value $f_* - \Theta_* \ge 0$ is called \textit{the duality gap}.

\begin{definition}
One says that \textit{the strong duality} holds or that the penalty function $F_c$ has \textit{the zero duality gap
property}, if $\Theta_* = f_*$, that is,
\[
  \sup_{c \ge 0} \inf_{x \in Q} F_c(x) = \inf_{x \in \Omega} f(x).
\]
\end{definition}

\begin{remark} \label{rmrk:ExactnessImpliesZeroDualGap}
Note that by Proposition~\ref{prp:ExactnessViaOptValues} the zero duality gap property is a necessary condition for 
the penalty function $F_c$ to be globally exact.
\end{remark}

We will need the following auxiliary result connecting the zero duality gap property with the lower semicontinuity of
the optimal value/perturbation function of the problem $(\mathcal{P})$ (see, e.g. 
\cite[Theorem~3.16]{Dolgopolik_ExPenII}).

\begin{theorem} \label{thrm:ZeroDualGap_vs_LSC_OptValFunc}
The penalty function $F_c$ has the zero duality gap property if and only if it is bounded below on $Q$ for some $c > 0$
and the optimal value function (with respect to the infeasibility measure $\varphi$)
\[
  \beta_{\varphi}(p) := \inf\Big\{ f(x) \Bigm| x \in Q \colon \varphi(x) \le p \Big\} \quad \forall p \ge 0
\]
is lower semicontinuous (lsc) at the origin.
\end{theorem}

In order to extend the localization principle to the case of well-posed infinite dimensional optimization problems we
need the following lemma on (approximate) minimizers of the penalty function. Various particular versions of this lemma 
are well-known and we include its proof only for the sake of completeness.

\begin{lemma} \label{lem:MinimizingSeq_PenaltyFunc}
Let a sequence $\{ \varepsilon_n \} \subset [0, + \infty)$ and a non-decreasing sequence 
$\{ c_n \} \subset (0, + \infty)$ be such that $\varepsilon_n \to 0$ and $c_n \to + \infty$ as $n \to \infty$ and 
the function $F_{c_0}$ is bounded below on $Q$. Then for any sequence $\{ x_n \} \subset Q$ of
$\varepsilon_n$-minimizers of $F_{c_n}$ on the set $Q$, that is,
\[
  F_{c_n}(x_n) \le \inf_{x \in Q} F_{c_n}(x) + \varepsilon_n \quad \forall n \in \mathbb{N},
\]
one has $\varphi(x_n) \to 0$ as $n \to \infty$ and
\begin{equation} \label{eq:PenMethod_ObjFuncLimits}
  \min\{ \liminf_{p \downarrow 0} \beta_{\varphi}(p), f_* \} 
  \le \liminf_{n \to \infty} f(x_n) \le \limsup_{n \to \infty} f(x_n) \le f_*,
\end{equation}
where $p \downarrow 0$ means that $p$ approaches $0$ from the right.
\end{lemma}

\begin{proof}
Note that the penalty function $F_{c_n}$ is bounded below on $Q$ for any $n \in \mathbb{N}$, since the sequence 
$\{ c_n \}$ is non-decreasing, the function $c \mapsto F_c(x)$ is non-decreasing for any $x \in X$, and $F_{c_0}$ is 
bounded below on $Q$ by our assumption. Therefore, for any $n \in \mathbb{N}$ the notion of $\varepsilon_n$-minimizer
of $F_{c_n}$ on $Q$ is correctly defined.

Suppose by contradiction that the infeasibility measure $\varphi(x_n)$ does not converge to zero. Then there exist 
$\delta > 0$ and a subsequence $\{ x_{n_k} \}$ such that $\varphi(x_{n_k}) \ge \delta$ for all $k \in \mathbb{N}$.
Observe that
\begin{align*}
  F_{c_{n_k}}(x_{n_k}) = f(x_{n_k}) + c_{n_k} \varphi(x_{n_k}) 
  &= F_{c_0}(x_{n_k}) + (c_{n_k} - c_0) \varphi(x_{n_k}) 
  \\
  &\ge \inf_{x \in Q} F_{c_0}(x) + (c_{n_k} - c_0) \delta
\end{align*}
for any $k \in \mathbb{N}$. The right-hand side of this inequality increases unboundedly as $k \to \infty$. Therefore,
$F_{c_{n_k}}(x_{n_k}) \to + \infty$ as $k \to \infty$, which is impossible due to the facts that (i) 
$F_{c_n}(x_n) \le f_* + \varepsilon_n$ by the weak duality and the definition of $x_n$, and (ii) $\varepsilon_n \to 0$
as $n \to \infty$. Thus, $\varphi(x_n) \to 0$ as $n \to \infty$.

Let us now prove inequalities \eqref{eq:PenMethod_ObjFuncLimits}. As was noted above, 
$F_{c_n}(x_n) \le f_* + \varepsilon_n$, which implies that $f(x_n) \le f_* + \varepsilon_n$ and, therefore, the upper
estimate in \eqref{eq:PenMethod_ObjFuncLimits} holds true.

To prove the lower estimate, choose any sequence $\{ x_{n_k} \}$ such that
\[
  \liminf_{n \to \infty} f(x_n) = \lim_{k \to \infty} f(x_{n_k})
\]
(at least one such sequence exists by the definition of the lower limit). If there exists a subsequence of the sequence
$\{ x_{n_k} \}$, which we denote again by $\{ x_{n_k} \}$, that is feasible for the problem $(\mathcal{P})$, then
$f(x_{n_k}) \ge f_*$ for all $k \in \mathbb{N}$ and
\[
  \liminf_{n \to \infty} f(x_n) = \lim_{k \to \infty} f(x_{n_k}) \ge f_*,
\]
that is, the lower estimate in \eqref{eq:PenMethod_ObjFuncLimits} holds true. 

Let us now consider the case when there does not exists a subsequence of the sequence $\{ x_{n_k} \}$ that is feasible
for the problem $(\mathcal{P})$. Denote $p_n = \varphi(x_n)$ for all $n \in \mathbb{N}$. Without loss of generality one
can suppose that the subsequence $\{ x_{n_k} \}$ is infeasible and, therefore, $p_{n_k} = \varphi(x_{n_k}) > 0$ for all
$k \in \mathbb{N}$. In addition, $p_n = \varphi(x_n) \to 0$ as $n \to \infty$, as was shown above.

By the definition of the optimal value function $f(x_{n_k}) \ge \beta_{\varphi}(p_{n_k})$ for all $k \in \mathbb{N}$. 
Consequently, one has
\[
  \liminf_{n \to \infty} f(x_n) = \lim_{k \to \infty} f(x_{n_k}) 
  \ge \liminf_{k \to \infty} \beta_{\varphi}(p_{n_k})
  \ge \liminf_{p \downarrow 0} \beta_{\varphi}(p),
\]
where the last inequality follows from the fact that $p_{n_k} \to 0$ as $k \to \infty$. Thus, the lower estimate in
\eqref{eq:PenMethod_ObjFuncLimits} holds true in both cases.
\end{proof}

Now we can prove new necessary and sufficient conditions for the global exactness of penalty functions for well-posed
problems. To conveniently formulate these conditions we will use the notions of the least exact penalty parameter 
and the radius of local exactness introduced in Section~\ref{sect:ExactPenalty_Prelim}.

\begin{theorem} \label{thrm:Exactness_for_LevitinPolyakGeneralizedWP}
Let the problem $(\mathcal{P})$ be weakly Levitin-Polyak well-posed with respect to the infeasibility measure $\varphi$.
Then the penalty function $F_c$ is globally exact if and only if the following three conditions hold true:
\begin{enumerate}
\item{$F_c$ has the zero duality gap property;}

\item{there exists $c_0 \ge 0$ such that $F_c$ attains a global minimum on $Q$ for any $c \ge c_0$;}

\item{$F_c$ is uniformly locally exact at globally optimal solutions of the problem $(\mathcal{P})$, that is, $F_c$ is
locally exact at every point $x_* \in \argmin(\mathcal{P})$ and there exist $0 < c_* < + \infty$ and 
$0 < r_* < + \infty$ such that $c_*(x_*) < c_*$ and $r_*(x_*, c_*) \ge r_*$ for any $x_* \in \argmin(\mathcal{P})$.
}
\end{enumerate}
\end{theorem}

\begin{proof}
If the penalty function $F_c$ is globally exact, then, as was noted in Remark~\ref{rmrk:ExactnessImpliesZeroDualGap}, 
it has the zero duality gap property. Furthermore, it is uniformly locally exact at every globally optimal solution of
the problem $(\mathcal{P})$ with any $c_* > c_*(f, \varphi)$ and any $r_* > 0$ by the definition of the global
exactness. In addition, for any $c > c_*(f, \varphi)$ the penalty function $F_c$ attains a global minimum on the set $Q$
at any globally optimal solution of the problem $(\mathcal{P})$.

Let us prove the converse statement. Suppose that the penalty function satisfies the three conditions from the 
formulation of the theorem. Choose any increasing unbounded sequence $\{ c_n \} \subset [0, + \infty)$ with $c_0 \ge 0$
from the second condition. By this condition for any $n \in \mathbb{N}$ the function $F_{c_n}$ attains a global minimum
on $Q$ at some point $x_n$.

By Lemma~\ref{lem:MinimizingSeq_PenaltyFunc} one has $\varphi(x_n) \to 0$ as $n \to \infty$ and inequalities
\eqref{eq:PenMethod_ObjFuncLimits} hold true. Moreover, by Theorem~\ref{thrm:ZeroDualGap_vs_LSC_OptValFunc} and the
first condition of the theorem inequalities \eqref{eq:PenMethod_ObjFuncLimits} imply that $f(x_n) \to f_*$ as 
$n \to \infty$. Consequently, by the weak Levitin-Polyak well-posedness with respect to $\varphi$ one has 
$\dist(x_n, \argmin(\mathcal{P})) \to 0$ as $n \to \infty$. Hence, in particular, there exists $n_0 \in \mathbb{N}$ 
such that
\[
  \dist(x_n, \argmin(\mathcal{P})) < r_*, \quad c_n \ge c_* \quad \forall n \ge n_0,
\]
where $c_* > 0$ and $r_* > 0$ are from the definition of the uniform local exactness (i.e. the third condition). Here
we used the fact that $\{ c_n \}$ is an increasing unbounded sequence.

Thus, for any $n \ge n_0$ there exists $x_* \in \argmin(\mathcal{P})$ such that $c_n > c_*(x_*)$ and 
$\| x_n - x_* \| < r_*(x_*, c_n)$, which by definitions means that 
\[
  \inf_{x \in Q} F_{c_n}(x) = F_{c_n}(x_n) \ge F_{c_n}(x_*) = f(x_*) = f_*.
\]
Therefore, the penalty function $F_c$ is globally exact by Proposition~\ref{prp:ExactnessViaOptValues}.
\end{proof}

\begin{corollary}[localization principle for penalty functions I]
Let the problem $(\mathcal{P})$ be Levitin-Polyak well-posed with respect to the infeasibility measure $\varphi$. Then
the penalty function $F_c$ is globally exact if and only if the following three conditions hold true:
\begin{enumerate}
\item{$F_c$ has the zero duality gap property;}

\item{there exists $c_0 \ge 0$ such that $F_c$ attains a global minimum on $Q$ for any $c \ge c_0$;}

\item{$F_c$ is locally exact at every globally optimal solution of the problem $(\mathcal{P})$.}
\end{enumerate}
\end{corollary}

\begin{proof}
By Proposition~\ref{prp:GenLevitinPolyakWellPosedness_2Versions} the problem $(\mathcal{P})$ is weakly Levitin-Polyak
well-posed with respect to the infeasibility measure $\varphi$ and the set $\argmin(\mathcal{P})$ is compact. With the
use of the compactness of this set and the last assumption of the corollary one can readily check that $F_c$ is
\textit{uniformly} locally exact at globally optimal solutions of the problem $(\mathcal{P})$. Hence applying
Theorem~\ref{thrm:Exactness_for_LevitinPolyakGeneralizedWP} one obtains the required result.
\end{proof}

In the case when $X$ is a reflexive Banach space, one can prove a significantly stronger version of 
Theorem~\ref{thrm:Exactness_for_LevitinPolyakGeneralizedWP}, which can be viewed as a direct extension of the
localization principle to the case of infinite dimensional well-posed constrained optimization problems (cf. the finite
dimensional version of this principle in \cite{Dolgopolik_ExPenI,Dolgopolik_ExPenII}).

\begin{theorem}[localization principle for penalty functions II] \label{thrm:LocPrinciple_ExPen_ReflCase}
Let $X$ be a reflexive Banach space, the set $Q$ be weakly sequentially closed (in particular, closed and convex), 
the functions $f$ and $\varphi$ be weakly sequentially lsc on the set $Q$, and the problem $(\mathcal{P})$ be weakly
Levitin-Polyak well-posed with respect to the infeasibility measure $\varphi$. Then the penalty function $F_c$ is
globally exact if and only if the following two conditions hold true:
\begin{enumerate}
\item{there exists $c_0 \ge 0$ such that the set $S_{c_0}(f_*) := \{ x \in Q \mid F_{c_0}(x) < f_* \}$ is either bounded
or empty;}

\item{$F_c$ is uniformly locally exact at globally optimal solutions of the problem $(\mathcal{P})$.}
\end{enumerate}
\end{theorem}

\begin{proof}
If $F_c$ is globally exact, then the set $S_c(f_*)$ is empty for any $c \ge c_*(f, \varphi)$ by
Proposition~\ref{prp:ExactnessViaOptValues} and, as was shown in the proof of
Theorem~\ref{thrm:Exactness_for_LevitinPolyakGeneralizedWP}, $F_c$ is uniformly locally exact at globally optimal 
solutions of the problem $(\mathcal{P})$. Let us prove the converse statement.

Note that $S_c(f_*) \subseteq S_t(f_*)$ for any $c \ge t \ge 0$ due to the fact that the function $c \mapsto F_c(x)$ is
nondecreasing for any $x \in X$. Let $c_0$ be from the formulation of the theorem. If there exists $c \ge c_0$ such that
the set $S_c(f_*) \subseteq S_{c_0}(f_*)$ is empty, then $\inf_{x \in Q} F_c(x) = f_*$ and the penalty function $F_c$ is
globally exact by Proposition~\ref{prp:ExactnessViaOptValues}. Therefore, one can suppose that the set $S_c(f_*)$ is
nonempty and bounded for any $c \ge c_0$.

Since the set $S_c(f_*)$ is nonempty and bounded for any $c \ge c_0$, the function $F_c$ attains a global minimum on 
the set $Q$ for any $c \ge c_0$ due to the facts that it is weakly sequentially lsc and the space $X$ is reflexive by 
the assumptions of the theorem. Choose an increasing unbounded sequence $\{ c_n \} \subset (0, + \infty)$ and let $x_n$
be a point of global minimum of $F_{c_n}$ on $Q$. Clearly, $x_n \in S_{c_n}(f_*)$ for any $n \in \mathbb{N}$.

By Lemma~\ref{lem:MinimizingSeq_PenaltyFunc} one has $\varphi(x_n) \to 0$ as $n \to \infty$. In addition,
$x_n \in S_{c_n}(f_*) \subseteq S_{c_0}(f_*)$ for any $n \in \mathbb{N}$, which implies that the sequence $\{ x_n \}$ is
bounded. Hence with the use of the reflexivity of the space $X$ one can extract a subsequence $\{ x_{n_k} \}$ weakly
converging to some $x_* \in Q$ (recall that the set $Q$ is weakly sequentially closed). Moreover, by virtue of the weak
sequential lower semicontinuity assumptions one has
\[
  0 = \liminf_{k \to \infty} \varphi(x_{n_k}) \ge \varphi(x_*), \quad
  \liminf_{k \to \infty} f(x_{n_k}) \ge f(x_*).
\]
Consequently, the point $x_*$ is feasible for the problem $(\mathcal{P})$ and $f(x_*) \ge f_*$. On the other hand,
observe that $f(x_n) < f_*$ due to the fact that $x_n \in S_{c_n}(f_*)$. Therefore, 
$\limsup_{n \to \infty} f(x_n) \le f_*$ and $f(x_{n_k}) \to f_*$ as $k \to \infty$.

Thus, the subsequence $\{ x_{n_k} \}$ satisfies the two following conditions: 
\begin{equation} \label{eq:LevitinPolyakSubseq}
  \lim_{k \to \infty} \varphi(x_{n_k}) = 0, \quad \lim_{k \to \infty} f(x_{n_k}) = f_*.
\end{equation}
Hence by applying the fact that the problem $(\mathcal{P})$ is weakly Levitin-Polyak well-posed with respect to the
infeasibility measure $\varphi$ one can conclude that $\dist(x_{n_k}, \argmin(\mathcal{P})) \to 0$ as $k \to \infty$.
Now arguing in the same way as in the proof of Theorem~\ref{thrm:Exactness_for_LevitinPolyakGeneralizedWP} and utilising
the uniform local exactness assumption one can readily prove the required result.
\end{proof}

\begin{remark}
One can provide many simple sufficient conditions for the set $S_c(f_*)$ to be either bounded or empty that can be
easily verified in many particular cases. For example, if one of the following conditions is satisfied:
\begin{enumerate}
\item{the set $Q$ is bounded,}

\item{the function $f$ is coercive on the set $Q$, that is, $f(x_n) \to + \infty$ for any sequence $\{ x_n \} \subset Q$
such that $\| x_n \| \to + \infty$ as $n \to \infty$,}

\item{the function $\varphi$ is coercive on $Q$,}

\item{the penalty function $F_c$ is coercive on $Q$ for some $c > 0$,}
\end{enumerate}
then the set $S_c(f_*)$ is either bounded or empty.
\end{remark}

Let us now turn to the case of problems that are Levitin-Polyak well-posed in the extended sense. Our aim is to show
that the localization principle admits an infinite dimensional extension to this class of problems as well. In order to
provide such extension we will need the following definition of \textit{nondegeneracy} of a penalty function introduced
by the author in \cite{Dolgopolik_ExPenI}.

\begin{definition}
{(i)~The penalty function $F_c$ is called \textit{nondegenerate}, if there exists $c_0 \ge 0$ such that for any 
$c \ge c_0$ the function $F_c$ attains a global minimum on the set $Q$ and there exists 
$x(c) \in \argmin_{x \in Q} F_c(x)$ such that the set $\{ x(c) \mid c \ge c_0 \}$ is bounded.
}

{(ii)~The penalty function $F_c$ is called \textit{strongly nondegenerate}, if there exists $c_0 \ge 0$ such that for
any $c \ge c_0$ the function $F_c$ attains a global minimum on the set $Q$ and there exists 
$x(c) \in \argmin_{x \in Q} F_c(x)$ such that the set $\{ x(c) \mid c \ge c_0 \}$ is bounded and 
$\dist(x(c), \Omega) \to 0$ as $c \to \infty$.
}

{(iii)~The penalty function $F_c$ is called \textit{nondegenerate in the extended sense}, if there exists $c_0 \ge 0$
such that for any $c \ge c_0$ the function $F_c$ attains a global minimum on the set $Q$ and there exists 
$x(c) \in \argmin_{x \in Q} F_c(x)$ such that $\dist(x(c), \Omega) \to 0$ as $c \to \infty$.
}
\end{definition}

\begin{remark}
{(i)~Some sufficient conditions for a penalty function function to be (strongly) nondegenerate can be found in
\cite{Dolgopolik_ExPenI}. See also \cite{Dolgopolik_ExPenI} for some examples showing that a nondegenerate penalty 
function might not be strongly nondegenerate in the general case.
}

\noindent{(ii)~One can easily verify that in Theorem~\ref{thrm:LocPrinciple_ExPen_ReflCase} the assumption that the set
$S_c(f_*)$ is either bounded or empty for any $c$ greater than some $c_0 \ge 0$ can be replaced by the assumption that
the penalty function $F_c$ is nondegenerate (see two different versions of the localization principle in 
\cite{Dolgopolik_ExPenII,Dolgopolik_UnifedI} for more details).
}
\end{remark}

As was shown in \cite{Dolgopolik_ExPenI} and can be readily verified directly, the strong nondegeneracy of a penalty 
function is a necessary condition for the global exactness of a penalty function and a necessary condition for 
the validity of the localization principle, but by itself the strong nondegeneracy is not sufficient to guarantee
neither global exactness nor the validity of the localization principle. Let us show that it becomes a necessary and
sufficient condition for the validity of the localization principle for problems that are Levitin-Polyak well-posed in
the extended sense.

\begin{theorem} \label{thrm:Exactness_for_LevitinPolyakExtendedWP}
Let the problem $(\mathcal{P})$ be weakly Levitin-Polyak well-posed in the extended sense. Then the penalty function
$F_c$ is globally exact if and only if the following three conditions hold true:
\begin{enumerate}
\item{$F_c$ has the zero duality gap property;}

\item{$F_c$ is non-degenerate in the extended sense;}

\item{$F_c$ is uniformly locally exact at globally optimal solutions of the problem $(\mathcal{P})$.}
\end{enumerate}
\end{theorem}

\begin{proof}
The necessity of the three conditions from the formulation of the theorem for the global exactness of the penalty
function can be readily verified directly (to check the nondegeneracy in the extended sense choose any 
$c_0 > c_*(f, \varphi)$ and set $x(c) = x_*$ for an $c \ge c_0$ and any fixed globally optimal solution $x_*$ of 
the problem $(\mathcal{P})$).

Let us show the sufficiency of the three conditions for the global exactness of the penalty function. Choose an
increasing unbounded sequence $\{ c_n \} \subset [0, + \infty)$, where $c_0 \ge 0$ is from the definition of
nondegeneracy in the extended sense. By this definition there exists a sequence $\{ x_n \}$ such that $x_n$ is a point
of global minimum of $F_{c_n}$ on $Q$ and $\dist(x_n, \Omega) \to 0$ as $n \to \infty$. Furthermore, from
Lemma~\ref{lem:MinimizingSeq_PenaltyFunc}, the fact that $F_c$ has the zero duality gap property, and
Theorem~\ref{thrm:ZeroDualGap_vs_LSC_OptValFunc} it follows that $f(x_n) \to f_*$ as $n \to \infty$. Therefore, by 
the weak Levitin-Polyak well-posedness in the extended sense one has $\dist(x_n, \argmin(\mathcal{P})) \to 0$ as 
$n \to \infty$. Utilising this fact and the uniform local exactness assumption and arguing in the same way as in 
the proof of Theorem~\ref{thrm:Exactness_for_LevitinPolyakGeneralizedWP} one can readily verify that the penalty
function $F_c$ is globally exact.
\end{proof}

\begin{corollary}
Let the problem $(\mathcal{P})$ be Levitin-Polyak well-posed in the extended sense. Then the penalty function $F_c$ is
globally exact if and only if the following three conditions hold true:
\begin{enumerate}
\item{$F_c$ has the zero duality gap property;}

\item{$F_c$ is non-degenerate in the extended sense;}

\item{$F_c$ is locally exact at every globally optimal solution of the problem $(\mathcal{P})$.}
\end{enumerate}
\end{corollary}

Let us also provide a stronger version of Theorem~\ref{thrm:Exactness_for_LevitinPolyakExtendedWP} in the reflexive
case.

\begin{theorem} \label{thrm:ExactPenalty_NonDeg}
Let $X$ be a reflexive Banach space, the set $Q$ be weakly sequentially closed, the functions $f$ and $\varphi$ be
weakly sequentially lsc, and the problem $(\mathcal{P})$ be weakly Levitin-Polyak well-posed in the extended sense. 
Then the penalty function $F_c$ is globally exact if and only if the following two conditions hold true:
\begin{enumerate}
\item{$F_c$ is strongly non-degenerate;}

\item{$F_c$ is uniformly locally exact at globally optimal solutions of the problem $(\mathcal{P})$.}
\end{enumerate}
\end{theorem}

\begin{proof}
The necessity of the two conditions of the theorem for the global exactness of the penalty function can be readily
verified directly. Let us check their sufficiency.

Choose any increasing unbounded sequence $\{ c_n \} \subset [0, + \infty)$ with $c_0 \ge 0$ being from the definition 
of strong nondegeneracy. By this definition for any $n \in \mathbb{N}$ there exists 
$x_n \in \argmin_{x \in Q} F_{c_n}(x)$ such that the sequence $\{ x_n \}$ is bounded and $\dist(x_n, \Omega) \to 0$ as
$n \to \infty$. Note also that by Lemma~\ref{lem:MinimizingSeq_PenaltyFunc} one has $\varphi(x_n) \to 0$ as 
$n \to \infty$.

Recall that the space $X$ is reflexive. Therefore, one can extract a subsequence $\{ x_{n_k} \}$ converging to some 
$x_* \in Q$. From the weak sequential lsc of $\varphi$ and the fact that $\varphi(x_n) \to 0$ as $n \to \infty$ it
follows that $\varphi(x_*) = 0$, that is, the point $x_*$ is feasible for the problem $(\mathcal{P})$. In turn, by the
definition of $x_n$ one obviously has $f(x_n) \le f_*$, which with the use of the weak sequential lsc of $f$ implies
that
\[
  f_* \ge \limsup_{k \to \infty} f(x_{n_k}) \ge \liminf_{k \to \infty} f(x_{n_k}) \ge f(x_*) \ge f_*
\]
(the last inequality follows from the fact that $x_*$ is feasible), that is, $f(x_{n_k}) \to f_*$ as $k \to \infty$.

Thus, for the subsequence $\{ x_{n_k} \}$ one has $\dist(x_{n_k}, \Omega) \to 0$ and $f(x_{n_k}) \to f_*$ as 
$k \to \infty$. Therefore $\dist(x_{n_k}, \argmin(\mathcal{P})) \to 0$ as $k \to \infty$, since the problem
$(\mathcal{P})$ is weakly Levitin-Polyak well-posed in the extended sense. Utilising this fact and repeating the same
argument as in the proof of Theorem~\ref{thrm:Exactness_for_LevitinPolyakGeneralizedWP} one can conclude that the
penalty function $F_c$ is globally exact.
\end{proof}

In the end of this section, let us note that neither the Levitin-Polyak well-posedness in the extended sense/with
respect to an infeasibility measure nor even the Tykhonov well-posedness are necessary for the exactness of a penalty
function, as the following simple example demonstrates.

\begin{example}
Let $X = Q = \mathbb{R}$, $M = [0, + \infty)$ and
\[
  f(x) = \begin{cases}
    x, & \text{if } x \le 1,
    \\
    1/x, & \text{if } x > 1.
  \end{cases}
\]
Note that $f_* = 0$ and $\argmin(\mathcal{P}) = \{ 0 \}$. Moreover, the problem $(\mathcal{P})$ is not Tykhonov 
well-posed, since for the sequence $x_n \equiv n$ one has $f(x_n) \to 0 = f_*$, but 
$\dist(x_n, \argmin(\mathcal{P})) \to + \infty$ as $n \to \infty$.

On the other hand, if one sets $\varphi(x) = \max\{ -x, 0 \}$, then $\varphi$ is an infeasibility measure for 
the constraint $x \in M$ on the set $Q = X$ and, in addition, the corresponding penalty function $F_c$ is globally exact
due to the fact that
\[
  F_c(x) = x + c \max\{ -x, 0 \} > 0 = f_* \quad \forall x \in (- \infty, 0) \enspace \forall c > 1.
\]
Note that in this example $c_*(f, \varphi) = 1$.
\end{example}

\begin{remark}
As was shown in this section, a suitable version of the Levitin-Polyak well-posedness allows one to extend the
localization principle to the infinite dimensional case. In \cite[Section~3.5]{Dolgopolik_ExPenI} a number of
counterexamples to the localization principle in the infinite dimensional case was presented. It is worth briefly
discussing these examples in the light of the results obtained in this section. The problem from
\cite[Example~3]{Dolgopolik_ExPenI} is Levitin-Polyak well-posed in the extended sense, but is not Levitin-Polyak
well-posed with respect to the chosen infeasibility measure, while the penalty function from this example is locally
exact at a unique globally optimal solution, but is not non-degenerate. The problem from
\cite[Example~4]{Dolgopolik_ExPenI} is Levitin-Polyak well-posed both in the extended sense and with respect to the
chosen infeasibility measure, while the penalty function from this example is strongly non-degenerate, but is not
uniformly locally exact at globally optimal solutions of the corresponding problem. Finally, the problem from
\cite[Example~5]{Dolgopolik_ExPenI} is not even Tykhonov well-posed in the extended sense, while the corresponding
penalty function is strongly non-degenerate and uniformly locally exact at globally optimal solutions of the
optimization problem from this example.

Thus, the counterexamples to the localization principle in the infinite dimensional case from 
\cite[Section~3.5]{Dolgopolik_ExPenI} can be viewed as examples showing that the claims of 
Theorems~\ref{thrm:Exactness_for_LevitinPolyakGeneralizedWP}--\ref{thrm:ExactPenalty_NonDeg} are no longer true,
if one of the assumptions of these theorems is not satisfied. 
\end{remark}

\section{Global saddle points of augmented Lagrangians for well-posed problems}
\label{sect:AugmentedLagrangians}

A suitable version of well-posedness of constrained optimization problems can be used not only to extend 
the localization principle for exact penalty functions to the infinite dimensional case, but also to provide first
verifiable necessary and sufficient conditions for the existence of global saddle points of augmented Lagrangians for
infinite dimensional cone constrained optimization problems that do not rely on any abstract on behaviour of the
problem with respect to perturbations. These necessary and sufficient conditions can be viewed as an extension of the
localization principle for global saddle points of augmented Lagrangians
\cite{Dolgopolik_AugmLagrMult,Dolgopolik_AugmLagrTheory,Dolgopolik_MinimaxExact} to the infinite dimensional case.

In this section we study augmented Lagrangians for the following cone constrained optimization problem:
\begin{equation} \label{prob:ConeConstrained}
  \min \: f(x) \quad \text{subject to} \quad G(x) \in K, \quad x \in Q,
\end{equation}
where $G \colon X \to Y$ is a given function, $Y$ is a normed space, and $K \subset Y$ is a closed convex cone. We
suppose that the feasible set of this problem is nonempty, its optimal value is finite, and there exists a globally
optimal solution of this problem.

\begin{remark}
Problem \eqref{prob:ConeConstrained} can be rewritten as the problem $(\mathcal{P})$, if one defines
$M = \{ x \in X \mid G(x) \in K \}$. Therefore, all definitions of well-posedness from Section~\ref{sect:WellPosedness}
can be naturally translated to the case of problem \eqref{prob:ConeConstrained}.
\end{remark}

Let us introduce an augmented Lagrangian for problem \eqref{prob:ConeConstrained}. To this end, we will use the same 
axiomatic augmented Lagrangian setting as in the author's earlier papers 
\cite{Dolgopolik_AugmLagrTheory,Dolgopolik_AugmLagrMethods}. Denote by $Y^*$ the topological dual space of $Y$, and let 
$\langle \cdot, \cdot \rangle$ be the duality pairing between $Y$ and its dual. Denote the polar cone of the cone $K$
by $K^* = \{ y^* \in Y^* \mid \langle y^*, y \rangle \le 0 \: \forall y \in K \}$.

Suppose that some function $\Phi \colon Y \times Y^* \times (0, + \infty) \to (- \infty, + \infty]$ is given. An
augmented Lagrangian for problem \eqref{prob:ConeConstrained} is defined as follows:
\[
  \mathscr{L}(x, \lambda, c) = f(x) + \Phi(G(x), \lambda, c).
\]
Here $\lambda \in Y^*$ is a multiplier and $c > 0$ is a penalty parameter.

Let $\Lambda \subset Y^*$ be some nonempty set of admissible multipliers. Typically, either $\Lambda = Y^*$ or 
$\Lambda = K^*$. Below we will utilise the following assumptions (axioms) on the function $\Phi$:
\begin{itemize}
\item[(A1)]{$\forall y \in K \: \forall \lambda \in \Lambda \: \forall c > 0$ one has $\Phi(y, \lambda, c) \le 0$;}

\item[(A2)]{$\forall y \in K \: \forall c > 0 \: \exists \lambda \in \Lambda$ such that $\Phi(y, \lambda, c) \ge 0$;}

\item[(A3)]{$\forall y \notin K \: \forall c > 0 \: \exists \lambda \in \Lambda$ such that 
$\lim\limits_{t \to + \infty} \Phi(y, t \lambda, c) = + \infty$;}

\item[(A4)]{$\forall y \in Y \: \forall \lambda \in \Lambda$ the function $c \mapsto \Phi(y, \lambda, c)$ is
non-decreasing;}

\item[(A5)]{$\forall \lambda \in \Lambda \: \forall c_0 > 0 \: \forall r > 0$ one has
\begin{multline*}
  \lim_{c \to + \infty} \inf\Big\{ \Phi(y, \lambda, c) - \Phi(y, \lambda, c_0) \Bigm| 
  \\
  y \in Y \colon \dist(y, K) \ge r, \: \Phi(y, \lambda, c_0) < + \infty \Big\} = + \infty.
\end{multline*}
}
\end{itemize}
One can readily verify that most particular augmented Lagrangians satisfy these assumptions for a suitably chosen set
$\Lambda$ of admissible multipliers (see numerous examples in 
\cite{Dolgopolik_AugmLagrTheory,Dolgopolik_AugmLagrMethods}). In particular, in the case when $Y$ is a Hilbert space 
they are satisfied for the Hestenes-Powell-Rockafellar augmented Lagrangian
\[
  \Phi(y, \lambda, c) = \frac{1}{2c} \Big[ \dist(\lambda + c y, K)^2 - \| \lambda \|^2 \Big]
\]
with $\Lambda = Y^*$.

Let us recall the definitions of global and local saddle points of augmented Lagrangians
(see \cite{RockafellarWets,ShapiroSun,Dolgopolik_AugmLagrMult,Dolgopolik_AugmLagrTheory}).

\begin{definition} \label{def:GlobalSaddlePoint}
A pair $(x_*, \lambda_*) \in Q \times \Lambda$ is called \textit{a global saddle point} of the augmented Lagrangian
$\mathscr{L}(x, \lambda, c)$, if there exists $c^* > 0$ such that
\[
  \sup_{\lambda \in \Lambda} \mathscr{L}(x_*, \lambda, c) \le \mathscr{L}(x_*, \lambda_*, c)
  \le \inf_{x \in Q} \mathscr{L}(x, \lambda_*, c) < + \infty \quad \forall c \ge c^*.
\]
The infimum of all such $c^*$ is denoted by $c^*(x_*, \lambda_*)$ and is called 
\textit{the least exact penalty parameter} for the global saddle point $(x_*, \lambda_*)$.
\end{definition}

\begin{definition}
A pair $(x_*, \lambda_*) \in Q \times \Lambda$ is called \textit{a local saddle point} of the augmented Lagrangian
$\mathscr{L}(x, \lambda, c)$, if there exist $c^* > 0$ and a neighbourhood $U$ of $x_*$ such that
\begin{equation} \label{eq:LocalSaddlePoint}
  \sup_{\lambda \in \Lambda} \mathscr{L}(x_*, \lambda, c) \le \mathscr{L}(x_*, \lambda_*, c)
  \le \inf_{x \in U \cap Q} \mathscr{L}(x, \lambda_*, c) < + \infty \quad \forall c \ge c^*.
\end{equation}
The infimum of all such $c^*$ is denoted by $c_{loc}^*(x_*, \lambda_*)$ and is called 
\textit{the least exact penalty parameter} for the local saddle point $(x_*, \lambda_*)$.
\end{definition}

Let $(x_*, \lambda_*)$ be a local saddle point of the augmented Lagrangian $\mathscr{L}(\cdot)$. Then
there exist $r > 0$ and $c_* > 0$ such that inequalities \eqref{eq:LocalSaddlePoint} are satisfied with
$U = B(x_*, r)$. Denote by $r_*(x_*, \lambda_*, c_*)$ the supremum of all $r > 0$ for which inequalities 
\eqref{eq:LocalSaddlePoint} are satisfied with $U = B(x_*, r)$. The quantity $r_*(x_*, \lambda_*, c_*)$ has the same
meaning as the radius of local exactness of a penalty function introduced in Section~\ref{sect:ExactPenalty_Prelim}
We will use it below to extend the assumption on the uniform local exactness of a penalty function from 
Theorem~\ref{thrm:Exactness_for_LevitinPolyakGeneralizedWP} to the case of a collection of local saddle points
of an augmented Lagrangian. 

Below we will need the two following auxiliary results on saddle points of augmented Lagrangians. The first one follows
directly from \cite[Remark~6]{Dolgopolik_AugmLagrTheory}, while the second one follows from 
\cite[Proposition~1 and Corollary~2]{Dolgopolik_AugmLagrTheory} and \cite[Corollary~6]{Dolgopolik_AugmLagrMethods}.

\begin{lemma} \label{lem:LSP_AugmLagrValue}
Let assumptions $(A1)$--$(A3)$ hold true and $(x_*, \lambda_*)$ be a local saddle point of $\mathscr{L}(\cdot)$. Then 
$\mathscr{L}(x_*, \lambda_*, c) = f(x_*)$ for any $c > c^*_{loc}(x_*, \lambda_*)$.
\end{lemma}

\begin{lemma} \label{lem:GlobalSP}
Let assumptions $(A1)$--$(A4)$ be satisfied and $(x_*, \lambda_*)$ be a global saddle point of $\mathscr{L}(\cdot)$.
Then the point $x_*$ is a globally optimal solution of problem \eqref{prob:ConeConstrained} and for any other globally
optimal solution $z_*$ of this problem the pair $(z_*, \lambda_*)$ is also a global saddle point of $\mathscr{L}(\cdot)$
and $c^*(z_*, \lambda_*) = c^*(x_*, \lambda_*)$.
\end{lemma}

Our aim is to provide verifiable sufficient conditions for the existence of a global saddle point of the augmented
Lagrangian $\mathscr{L}(\cdot)$ in the case when problem \eqref{prob:ConeConstrained} is well-posed in some suitable
sense. To this end, we will need the following auxiliary result, which can be viewed as an extension of
Lemma~\ref{lem:MinimizingSeq_PenaltyFunc} to the case of augmented Lagrangians for cone constrained
optimization problems. Denote by
\[
  \beta(p) = \inf\Big\{ f(x) \Bigm| x \in Q \colon G(x) - p \in K \Big\} \quad \forall p \in Y
\]
\textit{the optimal value function} (\textit{the perturbation function}) of problem \eqref{prob:ConeConstrained}.

\begin{lemma} \label{lem:AugmLagr_MinimizingSeqLimits}
Let assumptions $(A1)$, $(A4)$, $(A5)$ hold true. Suppose that a multiplier $\lambda \in \Lambda$, a sequence 
$\{ \varepsilon_n \} \subset [0, + \infty)$, and a non-decreasing sequence $\{ c_n \} \subset (0, + \infty)$ are such
that $\varepsilon_n \to 0$ and $c_n \to + \infty$ as $n \to \infty$ and the function $\mathscr{L}(\cdot, \lambda, c_0)$
is bounded below on $Q$. Then for any sequence $\{ x_n \} \subset Q$ of $\varepsilon_n$-minimizers 
of $\mathscr{L}(\cdot, \lambda, c_n)$ on the set $Q$, that is,
\[
  \mathscr{L}(x_n, \lambda, c_n) \le \inf_{x \in Q} \mathscr{L}(x, \lambda, c_n) + \varepsilon_n 
  \quad \forall n \in \mathbb{N},
\]
one has $\dist(G(x_n), K) \to 0$ as $n \to \infty$ and
\begin{equation} \label{eq:ALMethod_ObjFuncLimits}
  \min\{ \liminf_{p \to 0} \beta(p), f_* \} \le \liminf_{n \to \infty} f(x_n) \le \limsup_{n \to \infty} f(x_n) \le f_*,
\end{equation}
where $f_*$ is the optimal value of problem \eqref{prob:ConeConstrained}.
\end{lemma}

\begin{proof}
Firstly, note that by assumption $(A4)$ the function $c \mapsto \mathscr{L}(x, \lambda, c)$ is non-decreasing for any 
$x \in X$. Hence taking into account the facts that the function $\mathscr{L}(\cdot, \lambda, c_0)$ is bounded below on
$Q$ and the sequence $\{ c_n \}$ is non-decreasing one can conclude that for any $n \in \mathbb{N}$ the function
$\mathscr{L}(\cdot, \lambda, c_n)$ is bounded below on $Q$ and the notion of $\varepsilon_n$-minimizer of this function
on the set $Q$ is correctly defined for any $\varepsilon_n \ge 0$.

By assumption $(A1)$ for any point $x \in Q$ that is feasible for problem \eqref{prob:ConeConstrained} one has 
$\mathscr{L}(x, \lambda, c_n) \le f(x)$. Therefore
\begin{align*}
  \mathscr{L}(x_n, \lambda, c_n) &\le \inf_{x \in Q} \mathscr{L}(x, \lambda, c_n) + \varepsilon_n
  \le \inf_{x \in \Omega} \mathscr{L}(x, \lambda, c_n) + \varepsilon_n
  \\
  &\le \inf_{x \in \Omega} f(x) + \varepsilon_n = f_* + \varepsilon_n,
\end{align*}
where $\Omega$ is the feasible region of problem \eqref{prob:ConeConstrained}. Consequently, the upper estimate in 
\eqref{eq:ALMethod_ObjFuncLimits} holds true.

To prove the lower estimate, fix any $r > 0$ and $x \in Q$ satisfying the inequality $\dist(G(x), K) \ge r$. Note that 
for any $n \in \mathbb{N}$ one has
\begin{multline*}
  \mathscr{L}(x, \lambda, c_n) = \mathscr{L}(x, \lambda, c_0) + \Phi(G(x), \lambda, c_n) - \Phi(G(x), \lambda, c_0)
  \ge \inf_{z \in Q} \mathscr{L}(z, \lambda, c_0) 
  \\
  + \inf\Big\{ \Phi(y, \lambda, c_n) - \Phi(y, \lambda, c_0) \Bigm| 
  y \in Y \colon \dist(y, K) \ge r, \: \Phi(y, \lambda, c_0) < + \infty \Big\},
\end{multline*}
if $\Phi(G(x), \lambda, c_n) < + \infty$, and $\mathscr{L}(x, \lambda, c_n) = + \infty$ otherwise. The right-hand side
of this inequality converges to $+ \infty$ as $n \to \infty$ by assumption $(A5)$. On the other hand, as was noted
above, $\mathscr{L}(x_n, \lambda, c_n) \le f_* + \varepsilon_n$ for any $n \in \mathbb{N}$. Therefore, 
$\dist(G(x_n), K) \le r$ for any sufficiently large $n$. Since $r > 0$ was chosen arbitrarily, one can conclude that
$\dist(G(x_n), K) \to 0$ as $n \to \infty$.

Choose any subsequence $\{ x_{n_k} \}$ such that
\[
  \liminf_{n \to \infty} f(x_n) = \lim_{k \to \infty} f(x_{n_k}).
\]
At least one such subsequence exists by the definition of the lower limit. Replacing, if necessary, the sequence 
$\{ x_{n_k} \}$ with its subsequence, one can suppose that the sequence $\{ x_{n_k} \}$ is either feasible for 
the problem $(\mathcal{P})$ or $G(x_{n_k}) \notin K$ for all $k \in \mathbb{N}$.

If the subsequence $\{ x_{n_k} \}$ is feasible for the problem $(\mathcal{P})$, then $f(x_{n_k}) \ge f_*$ for all 
$k \in \mathbb{N}$ and 
\[
  \liminf_{n \to \infty} f(x_n) = \lim_{k \to \infty} f(x_{n_k}) \ge f_*.
\]
Suppose now that $G(x_{n_k}) \notin K$ for all $k \in \mathbb{N}$. Let $\{ p_k \} \subset Y$ be any sequence such that
$G(x_{n_k}) - p_k \in K$ for all $k \in \mathbb{N}$ and $\| p_k \| \to 0$ as $k \to \infty$. At least one such sequence
exists due to the fact that $\dist(G(x_n), K) \to 0$ as $n \to \infty$. By the definition of the optimal value function
$\beta$ one gets
\[
  \liminf_{n \to \infty} f(x_n) =  \lim_{k \to \infty} f(x_{n_k}) 
  \ge \liminf_{k \to \infty} \beta(p_k) \ge \liminf_{p \to 0} \beta(p).
\]
Combining the two cases together one obtains that the lower estimate in \eqref{eq:ALMethod_ObjFuncLimits} is satisfied.
\end{proof}

Let us now provide new necessary and sufficient conditions for the existence of global saddle points of augmented
Lagrangians for well-posed cone constrained optimization problems. 

\begin{theorem} \label{thrm:AugmLagr_GSP_MainExistenceThrm}
Let problem \eqref{prob:ConeConstrained} be weakly Levitin-Polyak well-posed with respect to the infeasibility measure 
$\varphi(\cdot) = \dist(G(\cdot), K)$ and the optimal value function of this problem be lsc at the origin. Suppose also
that assumptions $(A1)$--$(A5)$ are satisfied. Then a global saddle point of the augmented Lagrangian
$\mathscr{L}(\cdot)$ exists if and only if there exists a multiplier $\lambda_* \in \Lambda$ satisfying the two 
following conditions:
\begin{enumerate}
\item{there exists $c_0 \ge 0$ such that the function $\mathscr{L}(\cdot, \lambda_*, c)$ attains a global minimum on $Q$
for any $c \ge c_0$;}

\item{for any globally optimal solution $x_*$ of problem \eqref{prob:ConeConstrained} the pair $(x_*, \lambda_*)$ is 
a local saddle point of the augmented Lagrangian $\mathscr{L}(\cdot)$ and there exist $0 < c_* < + \infty$ and 
$0 < r_* < + \infty$ such that $c_{loc}^*(x_*, \lambda_*) < c_*$ $r_*(x_*, \lambda_*, c_*) \ge r_*$ for any 
globally optimal solution $x_*$ of problem \eqref{prob:ConeConstrained}.
}
\end{enumerate}
\end{theorem}

\begin{proof}
Suppose that there exists a global saddle point of the augmented Lagrangian $\mathscr{L}(\cdot)$. Denote it by 
$(x_*, \lambda_*)$. Then by Definition~\ref{def:GlobalSaddlePoint} for any $c > c^*(x_*, \lambda_*)$ the function
$\mathscr{L}(\cdot, \lambda_*, c)$ attains a global minimum on $Q$ at the point $x_*$, that is, the first assumption 
of the theorem is satisfied. Furthermore, by Lemma~\ref{lem:GlobalSP} for any globally optimal solution $z_*$ of
problem~\eqref{prob:ConeConstrained} the pair $(z_*, \lambda_*)$ is also a global saddle point of the augmented
Lagrangian and $c^*(x_*, \lambda_*) = c^*(z_*, \lambda_*)$. Therefore, the second assumption of the theorem is satisfied
for any $c_* > c^*(x_*, \lambda_*)$ and any $r_* > 0$.

Let us prove the converse statement. Suppose that the two assumptions of the theorem are valid. Choose any unboundedly
increasing sequence $\{ c_n \} \subset (0, + \infty)$ with $c_0$ being from the first assumption. By this assumption 
the function $\mathscr{L}(\cdot, \lambda_*, c_n)$ attains a global minimum on $Q$ at some point $x_n \in Q$. By
Lemma~\ref{lem:AugmLagr_MinimizingSeqLimits} one has $\dist(G(x_n), K) \to 0$ and $f(x_n) \to f_*$ as $n \to \infty$ 
due to the fact that the optimal value function $\beta$ is lsc at the origin by the assumption of the theorem.

Since by our assumption problem \eqref{prob:ConeConstrained} is weakly Levitin-Polyak well-posed with respect to the
infeasibility measure $\varphi(\cdot) = \dist(G(\cdot), K)$, one can conclude that $\dist(x_n, \Omega_*) \to 0$ as 
$n \to \infty$, where $\Omega_*$ is the set of globally optimal solutions of problem \eqref{prob:ConeConstrained}.
Therefore, there exists $n_0 \in \mathbb{N}$ such that
\begin{equation} \label{eq:AugmLagrMinimizersConvergToOptSet}
  \dist(x_n, \Omega_*) < r_*, \quad c_n \ge c_*, \quad \forall n \ge n_0,
\end{equation}
where $c_*$ and $r_*$ are from the second assumption of the theorem.

Fix any $n \ge n_0$. From \eqref{eq:AugmLagrMinimizersConvergToOptSet} it follows that there exists $x_* \in \Omega_*$
such that $\| x_n - x_* \| < r_*$ and, therefore, by the second assumption of the theorem and the definition of $x_n$
one has
\[
  \mathscr{L}(x_*, \lambda_*, c_n) \le \inf_{x \in B(x_*, r_*) \cap Q} \mathscr{L}(x, \lambda_*, c_n) 
  \le \mathscr{L}(x_n, \lambda_*, c_n) = \inf_{x \in Q} \mathscr{L}(x, \lambda_*, c_n).  
\]
With the use of Lemma~\ref{lem:LSP_AugmLagrValue} and assumption $(A4)$ one gets
\[
  \mathscr{L}(x_*, \lambda_*, c) = f(x_*) = \mathscr{L}(x_*, \lambda_*, c_n) 
  \le \inf_{x \in Q} \mathscr{L}(x, \lambda_*, c_n) \le \inf_{x \in Q} \mathscr{L}(x, \lambda_*, c)
\]
for all $c \ge c_n$. Moreover, from \eqref{eq:AugmLagrMinimizersConvergToOptSet}, the second assumption of the theorem
and the definition of local saddle point one has
\[
  \sup_{\lambda \in \Lambda} \mathscr{L}(x_*, \lambda, c) \le \mathscr{L}(x_*, \lambda_*, c) 
  \quad \forall c \ge c_n. 
\]
Thus, for any $c \ge c_n$ one has
\[
  \sup_{\lambda \in \Lambda} \mathscr{L}(x_*, \lambda, c) \le \mathscr{L}(x_*, \lambda_*, c)
  \le \inf_{x \in Q} \mathscr{L}(x, \lambda_*, c),
\]
which means that $(x_*, \lambda_*)$ is a global saddle point of the augmented Lagrangian with 
$c^*(x_*, \lambda_*) \le c_n$.
\end{proof}

\begin{corollary}[localization principle for global saddle points] \label{crlr:GSP_LocPrinciple}
Let problem \eqref{prob:ConeConstrained} be Levitin-Polyak well-posed with respect to the infeasibility measure 
$\varphi(\cdot) = \dist(G(\cdot), K)$ and the optimal value function of this problem be lsc at the origin. Suppose also
that assumptions $(A1)$--$(A5)$ are satisfied. Then a global saddle point of the augmented Lagrangian
$\mathscr{L}(\cdot)$ exists if and only if there exists a multiplier $\lambda_* \in \Lambda$ satisfying the following
two conditions:
\begin{enumerate}
\item{there exists $c_0 \ge 0$ such that the function $\mathscr{L}(\cdot, \lambda_*, c)$ attains a global minimum on $Q$
for any $c \ge c_0$;}

\item{for any globally optimal solution $x_*$ of problem \eqref{prob:ConeConstrained} the pair $(x_*, \lambda_*)$ is a
local saddle point of the augmented Lagrangian $\mathscr{L}(\cdot)$.
}
\end{enumerate}
\end{corollary}

\begin{proof}
The necessity of the two conditions from the formulation of the corollary for the existence of a global saddle point
follows directly from Theorem~\ref{thrm:AugmLagr_GSP_MainExistenceThrm}. To prove the sufficiency of these conditions,
suppose that sequences $\{ c_n \} \subset (0, + \infty)$ and $\{ x_n \} \subset Q$ are the same as in the proof of
Theorem~\ref{thrm:AugmLagr_GSP_MainExistenceThrm}. Then by Lemma~\ref{lem:AugmLagr_MinimizingSeqLimits} one has
$\dist(G(x_n), K) \to 0$ and $f(x_n) \to f_*$ as $n \to \infty$. Hence with the use of the fact that problem
\eqref{prob:ConeConstrained} is Levitin-Polyak well-posed with respect to the infeasibility measure 
$\varphi(\cdot) = \dist(G(\cdot), K)$ one obtains that there exists a subsequence $\{ x_{n_k} \}$ converging to some
globally optimal solution $x_*$ of problem \eqref{prob:ConeConstrained}.

By the second condition of the corollary the pair $(x_*, \lambda_*)$ is a local saddle point of $\mathscr{L}(\cdot)$. 
By definition it means that there exist a neighbourhood $U$ of $x_*$ and $c_* > 0$ such that
\[
  \sup_{\lambda \in \Lambda} \mathscr{L}(x_*, \lambda, c) \le \mathscr{L}(x_*, \lambda_*, c) \le 
  \inf_{x \in Q \cap U} \mathscr{L}(x, \lambda_*, c) \quad \forall c \ge c_*,
\]
while by Lemma~\ref{lem:LSP_AugmLagrValue} one has $\mathscr{L}(x_*, \lambda_*, c) = f(x_*)$ for all $c \ge c_*$. 
Since the subsequence $\{ x_{n_k} \}$ converges to $x_*$ and the sequence $\{ c_n \}$ increases unboundedly, there
exists $k \in \mathbb{N}$ such that $x_{n_k} \in U$ and $c_{n_k} \ge c_*$. For any such $k$ one has
\[
  f(x_*) = \mathscr{L}(x_*, \lambda_*, c_{n_k}) \le \mathscr{L}(x_{n_k}, \lambda_*, c_{n_k}) 
  = \inf_{x \in Q} \mathscr{L}(x, \lambda_*, c_{n_k})
\]
Hence bearing in mind assumption $(A4)$ one can conclude that
\[
  \sup_{\lambda \in \Lambda} \mathscr{L}(x_*, \lambda, c) \le \mathscr{L}(x_*, \lambda_*, c)
  = f(x_*) \le \inf_{x \in Q} \mathscr{L}(x, \lambda_*, c_{n_k})
  \le \inf_{x \in Q} \mathscr{L}(x, \lambda_*, c)
\]
for any $c \ge c_{n_k}$, which means that $(x_*, \lambda_*)$ is a global saddle point of $\mathscr{L}(\cdot)$.
\end{proof}

\begin{remark}
It should be noted that under very mild assumptions on the function $\Phi(y, \lambda, c)$ the lower semicontinuity 
of the optimal value function $\beta$ is, in fact, equivalent to the absence of duality gap between problem 
\eqref{prob:ConeConstrained} and its augmented dual problem
\[
  \max_{\lambda, c} \Theta(\lambda, c) \quad \text{subject to } \lambda \in \Lambda, \: c > 0,
\]
where $\Theta(\lambda, c) = \inf_{x \in Q} \mathscr{L}(x, \lambda, c)$ is the augmented dual function 
(see \cite{Dolgopolik_AugmLagrMethods}). One can also show that the absence of duality gap is a necessary conditions
for the existence of a global saddle point \cite{Dolgopolik_AugmLagrTheory,Dolgopolik_AugmLagrMethods}. Therefore, under
some additional assumptions the lower semicontinuity of the optimal value function $\beta$ is, in fact, a necessary
condition for the existence of a global saddle point of the augmented Lagrangian $\mathscr{L}(\cdot)$.
\end{remark}

\begin{remark}
Theorem~\ref{thrm:AugmLagr_GSP_MainExistenceThrm} can be viewed as an extension of
Theorem~\ref{thrm:Exactness_for_LevitinPolyakGeneralizedWP} on global exactness of penalty function to the case of
global saddle points of augmented Lagrangians It should be noted that all other results on exactness of penalty function
from Section~\ref{sect:ExactPenaltyFunctions} (e.g. Theorems~\ref{thrm:LocPrinciple_ExPen_ReflCase} and
\ref{thrm:Exactness_for_LevitinPolyakExtendedWP}) can be extended to the case of global saddle points of augmented
Lagrangians as well. We do not present such straightforward extensions here for the sake of shortness.
\end{remark}

\begin{remark}
Let us comment on a counterexample to the localization principle for global saddle points of augmented Lagrangians in
the infinite dimensional case given in \cite[Example~4]{Dolgopolik_AugmLagrMult}. One can readily verify that 
the optimization problem from this example is not even Tykhonov well-posed in the extended sense, but all other
assumptions of Corollary~\ref{crlr:GSP_LocPrinciple} are satisfied. Thus, \cite[Example~4]{Dolgopolik_AugmLagrMult}
shows that Corollary~\ref{crlr:GSP_LocPrinciple} does not holds true for ill-posed problems in the general case, even if
all other assumptions of this corollary are satisfied.
\end{remark}

\bibliographystyle{abbrv}  
\bibliography{GlobalSaddlePoints_bibl}

\end{document}